\documentclass[final,onefignum,onetabnum]{siamart190516}

\usepackage[utf8]{inputenc}
\usepackage{enumitem}
\usepackage{amsmath,amssymb,amsfonts,xcolor}
\usepackage{thmtools, thm-restate}
\usepackage{algorithmic, algorithm}
\crefname{subsection}{subsection}{subsections}
\usepackage{epstopdf}
\ifpdf
  \DeclareGraphicsExtensions{.eps,.pdf,.png,.jpg}
\else
  \DeclareGraphicsExtensions{.eps}
\fi

\usepackage{graphicx,subcaption}
\usepackage{pgfplots}
\pgfplotsset{compat=1.12}
\usepackage{tikz}
\usetikzlibrary{shapes.geometric, arrows}
\usetikzlibrary{positioning}
\usetikzlibrary{calc}
\usetikzlibrary{matrix}
\usetikzlibrary{patterns}
\usetikzlibrary{trees}
\tikzstyle{seps} = [circle, minimum size = 1cm, draw = black, fill=none]
\tikzstyle{arrow} = [thick,->,>=stealth]
\tikzstyle{darrow} = [thick,<->,>=stealth]
\newdimen\nodeDist
\usepackage{booktabs}
\usepackage{multirow}
\usepackage{tabularx}

\renewcommand{\r}[1]{\textcolor{red}{#1}}

\headers{Hierarchical Orthogonal Factorization: Sparse Square Matrices}{A. Gnanasekaran, E. Darve}

\title{Hierarchical Orthogonal Factorization: Sparse Square Matrices\thanks{Submitted to the editors on \today{
\funding{ This work was partly funded by a grant from Sandia National Laboratories (Laboratory Directed Research and Development [LDRD]) entitled ``Hierarchical Low-rank Matrix Factorizations,'' and a grant from the National Aeronautics and Space Administration (NASA, agreement \#80NSSC18M0152).}}}}

\author{Abeynaya Gnanasekaran\thanks{Institute for Computational and Mathematical Engineering, Stanford University, CA
  (\email{abeynaya@stanford.edu}), (\email{darve@stanford.edu})}
\and Eric Darve\footnotemark[2]}

\usepackage{amsopn}

\usepackage{hyperref}
\usepackage{bbm}

\usepackage{microtype}
\begin{document}

\maketitle
\begin{abstract}
In this work, we develop a new fast algorithm, spaQR --- sparsified QR, for solving large, sparse linear systems. The key to our approach is using low-rank approximations to sparsify the separators in a Nested Dissection based Householder QR factorization. First, a modified version of Nested Dissection is used to identify interiors/separators and reorder the matrix. Then, classical Householder QR is used to factorize the interiors, going from the leaves to the root to the elimination tree. After every level of interior factorization, we sparsify the remaining separators by using low-rank approximations. This operation reduces the size of the separators without introducing any fill-in in the matrix. However, it introduces a small approximation error which can be controlled by the user. The resulting approximate factorization is stored as a sequence of sparse orthogonal and sparse upper-triangular factors.  Hence, it can be applied efficiently to solve linear systems. 
Additionally, we further improve the algorithm by using a block diagonal scaling. Then, we show a systematic analysis of the approximation error and effectiveness of the algorithm in solving linear systems. Finally, we perform numerical tests on benchmark unsymmetric problems to evaluate the performance of the algorithm. The factorization time scales as $\mathcal{O}(N \log N)$ and the solve time scales as $\mathcal{O}(N)$.    
\end{abstract}


\section{Introduction}

We are interested in solving large, sparse, unsymmetric linear systems, 
\[
Ax =b, \quad A \in \mathbb{R}^{N \times N}.
\]
Iterative methods are preferred for sparse linear systems as they depend only on matrix-vector products, which can be computed in $\mathcal{O}\big(\text{nnz}(A)\big)$ time. Popular examples include Krylov space methods such as CG~\cite{Hestenes&Stiefel:1952}, GMRES~\cite{Saad1986GMRESAG}, MINRES~\cite{citeulike:10745617}. However, iterative methods rarely work well without good preconditioners which are essential for fast convergence to the solution. 

A naive LU or QR factorization of the matrix can cost $\mathcal{O}(N^3)$ even for sparse matrices due to the fill-in introduced during the factorization. However, one can ignore some of the fill-in entries to get an ``incomplete'' factorization of the matrix, which can then be used as a preconditioner for solving the associated linear system. For example, preconditioners like Incomplete LU~\cite{Saad1994ILUTAD}, Incomplete QR~\cite{jennings, Saad1988PreconditioningTF} and Incomplete Cholesky~\cite{Manteuffel1980AnIF} limit fill-in based on thresholding and on a prescribed maximum number of non-zeros in a row/column. While such methods are common in literature, there are no convergence guarantees nor provable efficiency for these preconditioners. In practice, they can fail for a large number of problems~\cite{Chow1997ExperimentalSO}. However, better preconditioners can be built when additional information on the problem is available.

In the recent years, another class of preconditioners have been developed based on the observation that certain off-diagonal blocks of $A$ or $A^{-1}$ are numerically low-rank. The matrices that exhibit this property are termed Hierarchical ($\mathcal{H}$) matrices~\cite{hmatrix_2, hmatrix_1, hmatrix_3, Hackbusch2000ASH}. While these methods were originally developed for dense matrices, there have been efforts to extend these ideas to sparse matrices, especially matrices arising out of PDE discretizations. These efforts have been focused on incorporating fast $\mathcal{H}-$algebra with a nested dissection based multifrontal elimination~\cite{mumps, H_QR,C:LaBRI::CIMI15, Ghysels2016AnEM, blr_pastix,  Schmitz2012AFD, Xia2013EfficientSM, Xia2009SuperfastMM}. For instance, a matrix-vector product can be done in almost linear time when the dense fronts are represented using low-rank bases.  

In contrast, we focus on another approach: continually decrease the size of the nested dissection separators by applying a low-rank approximation. As the size of the separators are reduced at every step, the algorithm never deals with large dense fronts. Some examples of these fast hierarchical solvers are the Hierarchical Interpolative Factorization (HIF)~\cite{feliufaba2020hierarchical, Ho2016HierarchicalIF}, LoRaSp~\cite{lorasp1, lorasp2} and Sparsified Nested Dissection (spaND)~\cite{2019arXiv190102971C, klockiewicz2020second}. All three algorithms were developed to perform fast Cholesky factorization of symmetric positive definite matrices. HIF and spaND have been extended to perform a fast LU factorization on unsymmetric matrices~\cite{Ho2016HierarchicalIF}. However, LU is known to be unstable unless a robust pivoting strategy is used which can be difficult for sparse matrices. Current sparse direct solvers often rely on \textit{ad hoc} techniques such as ignoring small pivots and replacing them by some large value $\epsilon^{-1}$ or postponing the elimination, leading to significant fill-in and an increase in the computational cost. 

In this work, we propose a novel fast hierarchical solver to perform QR factorization on sparse, square matrices using low-rank approximations. The algorithm can be extended, with some changes, to solve sparse linear least-squares problems. This will be discussed in a future work. The use of orthogonal transformations in the QR decomposition ensures stability and allows for a more robust treatment of unsymmetric matrices.  The resulting approximate factorization can then be used as a preconditioner with GMRES to solve general linear systems. Specifically, our algorithm produces a sparse approximate factorization of $A$ in near linear time, such that, 
\[ A \approx QW = \prod_i Q_i \prod_j W_j \]
where each $Q_i$ is a sparse orthogonal matrix and $W_j$ is either sparse orthogonal or sparse upper triangular. While $W$ is not necessarily upper triangular, we still use the term ``fast QR solver'' as the algorithm is built on top of classical Householder QR. 

\subsection{Contribution}

We propose, implement, and provide theoretical guarantees on a novel QR algorithm for unsymmetric, sparse matrices with full-rank. We henceforth refer to the algorithm as spaQR, or Sparsified QR. Our algorithm is built upon the ideas of the spaND algorithm, which was originally developed for SPD matrices. However, the existence and intuition behind spaQR is more involved as explained in \Cref{spars_s} and \Cref{Related_chol}. We summarize our main contributions as follows:
\begin{itemize}
    \item We propose and implement a novel fast QR algorithm with tunable accuracy for sparse square matrices. 
    \item We provide a systematic analysis of the approximation error and effectiveness of the preconditioner.
    \item We implement an additional block diagonal scaling that significantly improves the error and effectiveness of the preconditioner. The improvements from scaling are shown both theoretically and numerically.  
    \item We show that the factorization time scales as $\mathcal{O}(N \log N)$ and the solve time as $\mathcal{O}(N)$, under some assumptions
    \item We perform numerical tests on benchmark unsymmetric problems.
    \item The C++ code for the algorithm is freely available for download and use at this \href{https://github.com/Abeynaya/spaQR_public}{link}. The benchmarks can be reproduced by running the scripts available in the repository. 
\end{itemize}

The rest of the paper is organized as follows. \Cref{Sec: Algo} introduces the algorithm and the block scaling. This is followed by theoretical guarantees on the approximation error, effectiveness of the preconditioner and the complexity of the algorithm in \Cref{theoretical_results_sec}. Numerical results are discussed in \Cref{benchmarks}. Finally, we discuss directions for future research. We also give some intuition behind the algorithm and different variants of the algorithm in \Cref{Related_chol}.

\section{Algorithm}
\label{Sec: Algo}

We begin with a discussion on classical sparse QR factorization based on Householder transformations and Nested Dissection, giving an overview on the fill-in generated during the factorization. This is followed by a high level overview of the spaQR algorithm, followed by a detailed discussion and a discussion on the block diagonal scaling. 
 
\subsection{Sparse QR}
\label{QR}

Consider the Householder-based QR factorization of a sparse matrix $A\in \mathbb{R}^{m\times n}$ with $m \geq n$. Let $A^{[k]}$ denote the product $H_k H_{k-1} \dots H_1A$, where $H_k$ is the $k$-th Householder matrix. The sparsity of row $k$ in $R$ (and $A^{[k]}(k:m,:)$) can be understood in relation to the sparsity of $A^{[k-1]}$. When column $k$ of $A^{[k-1]}$  is operated on, all the rows $r$, that have non-zero entries in that column are affected. We introduce fill-in (or modify the existing entries) in all columns $c$ such that $A^{[k-1]}_{rc} \ne 0$ for any $r$ such that $A^{[k-1]}_{rk} \ne 0$. This can be seen as interactions between distance 1 and distance 2 neighbors (ignoring the direction of the edges) of node $k$ in \Cref{fillin}. This is in contrast to performing Gaussian Elimination on a matrix A, where we only have new interactions between distance 1 neighbors. Thus, fill-in in Householder QR is higher compared to the fill-in in Cholesky or LU factorization of a matrix. However, if $A$ has full column rank then the QR decomposition of $A$ and the Cholesky decomposition of $A^TA$ are related. In particular, if $A^TA = LL^T$, then $L = R^T(1:n, 1:n)$~\cite{10.5555/248979}.

\begin{figure}[tbhp]
    \centering
    \scalebox{0.75}{
    \begin{tikzpicture}[node distance=2cm]
\node (s) [seps] {$k$}; 
\node (n1) at ($ (s) + (45:2) $) [seps] {$n_1$};
\node (n2) at ($ (s) + (-45:2) $) [seps] {$n_2$};
\node (p) [seps, right of=n1] {$p$};
\node (q) [seps, right of=n2] {$q$};
\draw [arrow] (s.45) -- (n1);
\draw [arrow] (s.-45) -- (n2);
\draw [arrow] (p) -- (n1);
\draw [arrow] (q) -- (n2);
\end{tikzpicture}}
\hspace{2cm}
\scalebox{0.75}{
\begin{tikzpicture}[node distance=2cm]
\node (s) [seps] {$k$}; 
\node (n1) at ($ (s) + (45:2) $) [seps] {$n_1$};
\node (n2) at ($ (s) + (-45:2) $) [seps] {$n_2$};
\node (p) [seps, right of=n1] {$p$};
\node (q) [seps, right of=n2] {$q$};
\draw [arrow] (p) -- (n1);
\draw [arrow] (q) -- (n2);
\draw [darrow, dashed, color=red] (n1) -- (n2);
\draw [arrow, dashed, color=red] (q) -- (n1);
\draw [arrow, dashed, color=red] (p) -- (n2);
\draw [arrow] (n1) -- (s);
\draw [arrow] (n2) -- (s);
\draw [arrow, dashed, color=red] (p) -- (s);
\draw [arrow, dashed, color=red] (q) -- (s);
\end{tikzpicture}}
\quad

\[
\begin{matrix}
     & k & n_1 & n_2 & p & q \\
k   & \star &      &      &   &   \\
n_1 & \star  & \star     &      & \star   &   \\
n_2 & \star  &      &  \star    &   & \star  
\end{matrix}
\hspace{3cm}
\begin{matrix}
     & k & n_1 & n_2 & p & q \\
k    & \star & \r{\times}     &  \r{\times}    & \r{\times}  & \r{\times}   \\
n_1 &   & \star     & \r{\times}     & \star   & \r{\times}  \\
n_2 &   &  \r{\times}    &  \star    & \r{\times}  & \star  
\end{matrix}\]
    \caption{The graph of a sample matrix shown before and after one step of householder transformation on column $k$. There is a directed edge from node $j$ to node $i$ in the graph if $A(i,j) \ne 0$.  The fill-in entries are represented by red $\times$ symbols and the corresponding edges are denoted by red dashed lines.}
    \label{fillin}
\end{figure}
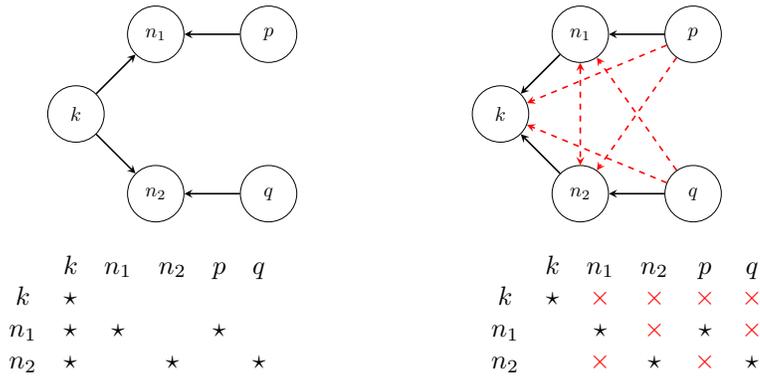

The relationship between the two factorizations allows us to extend the column reordering strategies developed for Cholesky to QR. The problem of finding an optimal permutation matrix $P$ for an SPD matrix $S$, such that the Cholesky factor of $PSP^T = LL^T$ has minimum fill-in is NP-hard. However, practical techniques based on heuristics have been developed and studied over the years. Some examples include minimum degree ordering, nested dissection, and Cuthill-McKee ordering. The reordering strategy that we use is Nested Dissection (ND) as it provides a convenient way to define separators and reinterpret the matrix as a block matrix. ND is a type of graph partitioning and works by recursively subdividing a graph while minimizing the number of edge cuts. 

Consider the sparse symmetric matrix $A^TA = S \in \mathbb{R}^{N \times N}$ and its graph $G_S = (V,E)$ where $V = \{1, 2, \dots, N\}$ and $E = \{(i,j): S_{ij}\ne 0\}$. ND works by finding vertex separators, which are groups of vertices that divide the graph into two disconnected components. \Cref{ND_seps} shows the vertex separators when recursively subdividing the graph three times. The process stops when the cluster sizes are small enough to be factored using a dense factorization scheme. 

The matrix factorization starts at the \textit{leaves}, which are the vertex \textit{clusters} at the last level (for example, $l=4$ in \Cref{ND_etree}) of the ND ordering. Once these are factorized, the factorization proceeds to the separators at the next lower level ($l=3$ in \Cref{ND_etree}) and continues to the top of the tree. This can be represented using an elimination tree as shown in \Cref{ND_etree}. The edges in the elimination tree indicate the dependencies between operations. Clusters at the same level can be operated on independently of one another. By factorizing from the leaves to the root of the elimination tree, we never create an edge (fill-in) between vertex clusters that are originally separated. The vertex separators obtained from the ND process on the matrix $A^TA$ provide a column partition for the matrix $A$, with the same fill-in guarantees. We discuss row partitioning ideas in \Cref{ord_clus}.

\begin{figure}[tbhp]
    \centering
    \begin{subfigure}[t]{0.35\textwidth}
        \centering
        \scalebox{0.75}{
        \begin{tikzpicture}
            \filldraw[fill = white, rounded corners, line width = 0.25mm] (0, 0) rectangle (3, 4) {};
           \filldraw[fill = darkgray, rounded corners, line width= 0.25mm] (1.325, 0) rectangle (1.625, 4) {};
            \filldraw[fill= gray, line width =0.25mm, rounded corners] (0, 2.5) rectangle (1.325, 2.75) {};
            \filldraw[fill= gray, line width =0.25mm, rounded corners] (1.625, 1.5) rectangle (3, 1.75) {};
            \filldraw[fill= lightgray, line width =0.25mm, rounded corners] (0.65, 0) rectangle (0.9, 2.5) {};
            \filldraw[fill= lightgray, line width =0.25mm, rounded corners] (0, 3.25) rectangle (1.325, 3.5) {};
            \filldraw[fill= lightgray, line width =0.25mm, rounded corners] (2.275, 1.75) rectangle (2.525, 4) {};
            \filldraw[fill= lightgray, line width =0.25mm, rounded corners] (1.625, 0.75) rectangle (3, 1) {};
        \end{tikzpicture}
        }
        \caption{Vertex separators}
        \label{ND_seps}
    \end{subfigure}%
   ~
    \begin{subfigure}[t]{0.6\textwidth}
        \centering
        \scalebox{0.75}{
        \begin{tikzpicture}[level distance=0.9cm,
  level 1/.style={sibling distance=3cm},
  level 2/.style={sibling distance=1.5cm},
  level 3/.style={sibling distance=0.75cm},
  edge from parent/.append style = {line width = 0.25mm}]
  \node (Root) [circle, draw=black, fill=darkgray] {}
    child {node [circle, draw=black, fill=gray]{}
      child {node [circle, draw=black, fill=lightgray]{}
        child {node [circle, draw=black]{}}
        child {node [circle, draw=black]{}}
      }
      child {node [circle, draw=black, fill=lightgray]{}
        child {node [circle, draw=black]{}}
        child {node [circle, draw=black]{}}
      }
    }
    child {node [circle, draw=black, fill=gray]{}
      child {node [circle, draw=black, fill=lightgray]{}
        child {node [circle, draw=black]{}}
        child {node [circle, draw=black]{}}
      }
      child {node [circle, draw=black, fill=lightgray]{}
        child {node [circle, draw=black]{}}
        child {node [circle, draw=black]{}}
      }
    };
    \begin{scope}[every node/.style={right}]
     \path (Root    -| Root-2-2-2) ++(5mm,0) node {$l=1$};
     \path (Root-1  -| Root-2-2-2) ++(5mm,0) node {$l=2$};
     \path (Root-1-1-| Root-2-2-2) ++(5mm,0) node {$l=3$};
     \path (Root-1-1-1-| Root-2-2-2) ++(5mm,0) node {$l=4$};
   \end{scope}

        \end{tikzpicture}}
        \caption{Elimination tree}
        \label{ND_etree}
    \end{subfigure}
    \caption{A four level nested dissection on an arbitrary graph. The figure on the left shows the vertex separators when recursively subdividing the graph and the figure on the right shows the corresponding elimination tree.}
    \label{ND}
\end{figure}
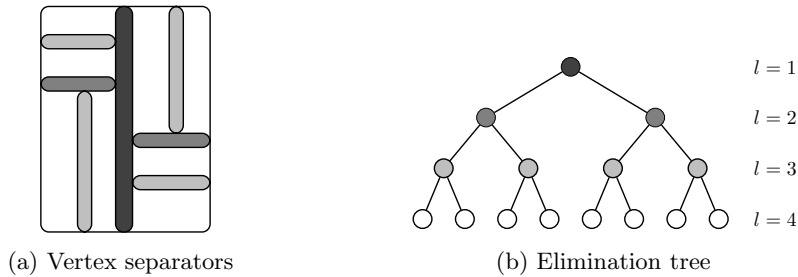

Nested Dissection ordering is usually used for elliptic partial differential equations discretized on 2D and 3D meshes. The cost of the Cholesky factorization on the reordered matrix reduces to $\mathcal{O}(N^{3/2})$ for 2D problems and $\mathcal{O}(N^2)$ for 3D problems, whereas the fill-in reduces to $\mathcal{O}(N \log N)$ in 2D and $\mathcal{O}(N^{3/4})$ in 3D~\cite{10.5555/248979}. 

Even with Nested Dissection, the fill-in is still significant. For 3D problems, the top separator has size $\mathcal{O}(N^{2/3})$ and its matrix block is dense when all its descendants are eliminated. Hence, the factorization of the top separator block will cost $\mathcal{O}(N^2)$. These arguments extend to the QR factorization, which has the same asymptotic cost. We can bring down the cost of performing QR on these problems to $\mathcal{O}(N)$ by `sparsifying' subsets of the separators as discussed next.

\subsection{Sparsified QR (spaQR)}
\label{spaQR_s}

The spaQR algorithm works by continually decreasing the size of a vertex separator in the trailing matrix by using a low-rank approximation of its neighbors. The algorithm alternates between factoring (block QR) the separators at a level $l$ and `sparsifying' the interfaces at all levels $l'>l$. 

We define an interface as a connected subset of a separator whose size is comparable to the diameter of the subdomains at that level. \Cref{interfaces_full} shows the distinction between separators and interfaces on a 3-level ND partition of a regular grid; \Cref{seps} shows the separators and \Cref{int} shows the interfaces. Denote the total number of levels as $L$ where the leaves correspond to $l = L$ and the root is at $l = 1$. Let $\hat{A}^l$ be the trailing matrix corresponding to level $1, 2, \dots, l$ of the matrix $A^{[l+1]} = H_{l+1}H_{l+2}\dots H_{L}A$, $\forall l< L$. Note that each of the householder matrices $H_k$ corresponds to a block reflector for the clusters at level $k$.

\begin{figure}[tbhp]
    \centering
    \begin{subfigure}[t]{0.25\textwidth}
        \centering
        \scalebox{0.75}{
        \begin{tikzpicture}
            \filldraw[fill = white, rounded corners, line width = 0.25mm] (0, 0) rectangle (3, 4) {};
            \filldraw[fill = gray, rounded corners, line width= 0.25mm] (1.325, 0) rectangle (1.625, 4) {};
            \filldraw[fill= lightgray, line width =0.25mm, rounded corners] (0, 3) rectangle (1.325, 3.25) {};
            \filldraw[fill= lightgray, line width=0.25mm, rounded corners] (1.625, 1) rectangle (3, 1.25) {};
        \end{tikzpicture}}
        \caption{Vertex separators}
        \label{seps}
    \end{subfigure}
    \begin{subfigure}[t]{0.25\textwidth}
        \centering
        \scalebox{0.75}{
        \begin{tikzpicture}
            \filldraw[fill = white, rounded corners, line width = 0.25mm] (0, 0) rectangle (3, 4) {};
            \filldraw[fill = gray, rounded corners, line width= 0.25mm] (1.325, 0) rectangle (1.625, 1) {};
            \filldraw[fill = gray, rounded corners, line width= 0.25mm] (1.325, 1) rectangle (1.625, 1.25) {};
            \filldraw[fill =gray, rounded corners, line width= 0.25mm] (1.325, 1.25) rectangle (1.625, 3) {};
            \filldraw[fill = gray, rounded corners, line width= 0.25mm] (1.325, 3) rectangle (1.625, 3.25) {};
            \filldraw[fill = gray, rounded corners, line width= 0.25mm] (1.325, 3.25) rectangle (1.625, 4) {};
            \filldraw[fill= lightgray, line width =0.25mm, rounded corners] (0, 3) rectangle (1.325, 3.25) {};
            \filldraw[fill= lightgray, line width=0.25mm, rounded corners] (1.625, 1) rectangle (3, 1.25) {};
        \end{tikzpicture}}
        \caption{Interfaces}
        \label{int}
    \end{subfigure}
    \caption{A three level nested dissection on an arbitrary graph. The figure on the left shows the usual nested dissection separators and the one on the right shows the interfaces.}
    \label{interfaces_full}
\end{figure}
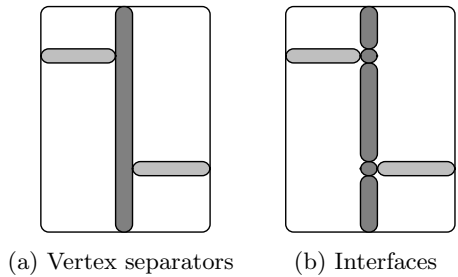

\[A^{[l+1]} = 
\begin{bmatrix}
R_{l+1:L,l+1:L} & R_{l+1:L,1:l} \\
& \hat{A}^l
\end{bmatrix}\]
where, $R_{l+1:L,l+1:L}$ is an upper-triangular block. The notation $R_{l+1:L,l+1:L}$ may appear confusing. Recall that $l=L$ corresponds to the leaf level in the tree (that is the ``top left'' part of the matrix), while $l=1$ is the top of the tree (this is the ``bottom right'' of the matrix). There is a slight inconsistency between the numbering of the levels in the tree ($l=1$ is the top) and the usual row/column numbering of the matrix (which starts at $l=L$ with our numbering). For consistency, we stick to indices associated with levels in the tree.

We can rewrite this as,
\[A^{[l+1]} = \begin{bmatrix}
I_{l+1:L, l+1:L} & \\
 & \hat{A}^{l}
\end{bmatrix}
\begin{bmatrix}
R_{l+1:L,l+1:L} & R_{l+1:L,1:l} \\
& I_{1:l,1:l}
\end{bmatrix}\]
and focus only on $\hat{A}^l$ (trailing matrix).

Let $p$ be a subset of the top ND separator (in dark grey) in \Cref{int} at the interface between two interiors (that have been eliminated) and let $n$ be all the nodes it's connected to ($\hat{A}^l_{np} \ne 0)$. Consider the submatrix of $\hat{A}^l$ corresponding to this interface $p$,
\[ \hat{A}^l_p = \begin{bmatrix}
\hat{A}^l_{pp} & \hat{A}^l_{pn} \\
\hat{A}^l_{np} & \hat{A}^l_{nn}
\end{bmatrix}\]
We work on the assumption that the off-diagonal blocks $\hat{A}_{np}^{l}$,  $\hat{A}_{pn}^l$ corresponding to an interface are low rank. We begin by computing a rank-revealing factorization of $\begin{bmatrix}
\hat{A}_{np}^{lT} & \sigma\hat{A}_{pp}^{lT}\hat{A}^l_{pn}
\end{bmatrix}$, for a constant $\sigma$ to be defined later. The two terms in the rank-revealing factorization are necessary for specific reasons. The first term $\hat{A}_{np}^{lT}$ is present to decouple a part of the interface $p$ from $n$. The second term $\sigma\hat{A}_{pp}^{lT}\hat{A}^l_{pn}$ ensures that the structure of the elimination tree is not broken by the sparsification. Since, the fill-in guarantees are directly related to the elimination tree, this ensures that we do not introduce additional non-zeros in the matrix as the algorithm proceeds. Alternately, we can think of it as finding an orthogonal transformation such that a subset of $p$ is decoupled from $n$ both during QR on $A$ and Cholesky on $A^TA$. More discussion on this connection to Cholesky is given in subsection \Cref{Related_chol}.

Begin by computing a low-rank approximation of,
\[\begin{bmatrix}
\hat{A}_{np}^{lT} & \sigma\hat{A}_{pp}^{lT}\hat{A}^l_{pn}
\end{bmatrix}  = Q_{pp}W_{pn} = \begin{bmatrix}
Q_{pf} & Q_{pc}
\end{bmatrix}\begin{bmatrix}
W_{fn} \\
W_{cn}
\end{bmatrix} \text{with } \|W_{fn}\|_{_2}=\mathcal{O}(\epsilon)\]
where, $\sigma$ is a scalar that will be defined later in \Cref{spars_s}. This gives us, 
\[ \begin{bmatrix}
\hat{A}^l_{pp} & \hat{A}^l_{pn} \\
\hat{A}^l_{np} & \hat{A}^l_{nn}
\end{bmatrix} \begin{bmatrix}
Q_{pp} & \\
& I
\end{bmatrix} = \begin{bmatrix}
\hat{A}^l_{ff} & \hat{A}^l_{fc}  & \hat{A}^l_{fn}  \\
\hat{A}^l_{cf}  & \hat{A}^l_{cc}  & \hat{A}^l_{cn}  \\
\mathcal{O}(\epsilon)  & W_{cn}^T & \hat{A}^l_{nn}
\end{bmatrix} \text{ where, } \hat{A}_{pn}^l = \begin{bmatrix}
\hat{A}^l_{fn} \\
\hat{A}^l_{cn}
\end{bmatrix}\]

The orthogonal transformation $Q$ splits the nodes in interface $p$ into `fine' $f$ and `coarse' $c$ nodes. Ignoring the $\mathcal{O}(\epsilon)$ terms and applying a block Householder transform on the columns of the $f$ block, 
\begin{align*}
   \begin{bmatrix}
H_{pf}^T & \\
& I
\end{bmatrix}\begin{bmatrix}
\hat{A}^l_{pp} & \hat{A}^l_{pn} \\
\hat{A}^l_{np} & \hat{A}^l_{nn}
\end{bmatrix} \begin{bmatrix}
Q_{pp} & \\
& I
\end{bmatrix} &= \begin{bmatrix}
R_{ff} & R_{fc} & \mathcal{O}(\epsilon)\\
& \Tilde{A}_{cc}^{l} & \Tilde{A}_{cn}^l \\
& W_{cn}^T & \hat{A}^l_{nn}
\end{bmatrix} \\
&= \begin{bmatrix}
I_f & & \\
& \Tilde{A}_{cc}^{l} & \Tilde{A}_{cn}^l \\
& W_{cn}^T & \hat{A}^l_{nn}
\end{bmatrix} \begin{bmatrix}
R_{ff} & R_{fc} & \mathcal{O}(\epsilon)\\
& I_c & \\ 
& & I_n
\end{bmatrix} 
\end{align*}

The $\mathcal{O}(\epsilon)$ terms are dropped. With this, the fine nodes are disconnected from the rest. Hence, the number of nodes in the interface $p$ has been reduced by $|f|$. In other words, interface $p$ has been sparsified.  We can once again focus on the trailing matrix and continue the algorithm. 

Following this procedure, we can sparsify all the remaining interfaces.  Detailed proofs (like why $R_{fn} = \mathcal{O}(\epsilon)$ and its significance) and discussion on why the sparsification  does not affect the elimination tree ordering (and hence the fill-in guarantees that come with it) are given in \Cref{spars_s}.

\begin{algorithm}
\caption{High level spaQR algorithm}
\begin{algorithmic}[1]
  \REQUIRE {Sparse matrix A, Maximum level L, Tolerance $\epsilon$}
   \STATE {Compute column and row partitioning of A, infer separators and interfaces (see \Cref{ord_clus})}
   \FORALL{$l=L, L-1, \dots 1$}
   \FORALL{Interiors $\mathcal{I}$ at level $l$}
        \STATE {Factorize $\mathcal{I}$ using block Householder (see \Cref{sparseQR_S})}
   \ENDFOR
   \FORALL{Interfaces $\mathcal{S}$ between interiors}
        \STATE {Sparsify $\mathcal{S}$ using tolerance $\epsilon$ (see \Cref{spaQR_s} and \Cref{spars_s})}
   \ENDFOR
   \ENDFOR
\end{algorithmic}
\label{highlevel_Algo}
\end{algorithm}

The spaQR algorithm alternates between factorization of the interiors at a level $l$ and sparsifying the interfaces at all levels $l'< l$. \Cref{highlevel_Algo} gives the high-level overview of spaQR. In the next few sections, we provide a detailed explanation on row/column reordering, defining interfaces, interior factorization and interface sparsification.

\subsection{Ordering and Clustering}
\label{ord_clus}

As we discussed earlier, Nested Dissection on the graph of $A^TA$ ($G_{A^TA}$) can be used to define the separators, which provides a column ordering for the matrix $A$. However, the cost of forming $A^TA$ is $\mathcal{O}(N^3)$ and is not preferred. Instead we use a hypergraph based partitioning technique that uses only the structure of $A$. The algorithm referred to as hypergraph-based unsymmetric nested dissection (HUND) developed in~\cite{Grigori_hypergraph-basedunsymmetric} is used for partitioning general matrices. Partitioning of hypergraphs is a well-studied problem and there are multiple software options like PaToH~\cite{atalyrek2011PaToHT}, hMetis~\cite{Karypis1998HmetisAH} and Zoltan~\cite{ZoltanHypergraphIPDPS06} to do the same. The problem of finding vertex separators in $A^TA$ is equivalent to finding hyperedge separators in $A$ as shown in~\cite{Catalyurek_hypergraph-partitioningbased, Grigori_hypergraph-basedunsymmetric, atalyrek1999HypergraphMF}. 

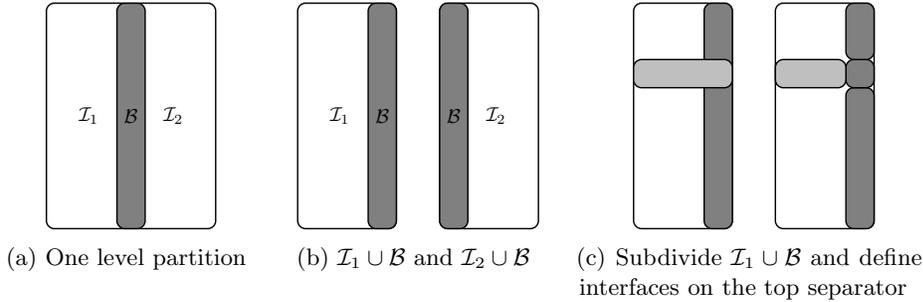
\begin{figure}[tbhp]
    \centering
    \begin{subfigure}[t]{0.25\textwidth}
        \centering
        \scalebox{0.75}{
        \begin{tikzpicture}
            \filldraw[fill = white, rounded corners, line width = 0.25mm] (0, 0) rectangle (3, 4) {};
            \filldraw[fill = gray, rounded corners, line width= 0.25mm] (1.25, 0) rectangle (1.75, 4) {};
            \node at (0.75,2) {$\mathcal{I}_1$};
            \node at (1.5,2) {$\mathcal{B}$};
            \node at (2.25,2) {$\mathcal{I}_2$};
        \end{tikzpicture}}
        \caption{One level partition}
    \end{subfigure}%
   ~
    \begin{subfigure}[t]{0.3\textwidth}
        \centering
        \scalebox{0.75}{
        \begin{tikzpicture}
            \filldraw[fill = white, rounded corners, line width = 0.25mm] (0, 0) rectangle (1.75, 4) {};
            \filldraw[fill = gray, rounded corners, line width= 0.25mm] (1.25, 0) rectangle (1.75, 4) {};
            \node at (0.75,2) {$\mathcal{I}_1$};
            \node at (1.5,2) {$\mathcal{B}$};
        \end{tikzpicture}
        }
        \quad
        \scalebox{0.75}{
        \begin{tikzpicture}
            \filldraw[fill = white, rounded corners, line width = 0.25mm] (0, 0) rectangle (1.75, 4) {};
            \filldraw[fill = gray, rounded corners, line width= 0.25mm] (0, 0) rectangle (0.5, 4) {};
            \node at (0.25,2) {$\mathcal{B}$};
            \node at (1,2) {$\mathcal{I}_2$};
        \end{tikzpicture}}
        \caption{$\mathcal{I}_1 \cup \mathcal{B}$ and $\mathcal{I}_2 \cup \mathcal{B}$ }
    \end{subfigure}%
   ~
    \begin{subfigure}[t]{0.35\textwidth}
        \centering
        \scalebox{0.75}{
        \begin{tikzpicture}
            \filldraw[fill = white, rounded corners, line width = 0.25mm] (0, 0) rectangle (1.75, 4) {};
            \filldraw[fill = gray, rounded corners, line width= 0.25mm] (1.25, 0) rectangle (1.75, 4) {};
            \filldraw[fill= lightgray, line width =0.25mm, rounded corners] (0, 2.5) rectangle (1.75, 3) {};
        \end{tikzpicture}}
        \quad
        \scalebox{0.75}{
        \begin{tikzpicture}
            \filldraw[fill = white, rounded corners, line width = 0.25mm] (0, 0) rectangle (1.75, 4) {};
            \filldraw[fill = gray, rounded corners, line width= 0.25mm] (1.25, 0) rectangle (1.75, 2.5) {};
            \filldraw[fill = gray, rounded corners, line width= 0.25mm] (1.25, 2.5) rectangle (1.75, 3) {};
            \filldraw[fill =gray, rounded corners, line width= 0.25mm] (1.25, 3) rectangle (1.75, 4) {};
            \filldraw[fill= lightgray, line width =0.25mm, rounded corners] (0, 2.5) rectangle (1.25, 3) {};
        \end{tikzpicture}}
        \caption{Subdivide $\mathcal{I}_1 \cup \mathcal{B}$ and define interfaces on the top separator}
    \end{subfigure}
    \caption{The first figure shows a one level partition of an arbitrary graph (hypergraph) using nested dissection (HUND). The next two figures depict the process of identifying the interfaces by subdividing $\mathcal{I}_1 \cup \mathcal{B}$.}
    \label{mnd}
\end{figure}

However, in addition to defining separators, we need a clustering of the unknowns in a separator to define interfaces. In SpaND~\cite{2019arXiv190102971C}, the technique of modified nested dissection is developed to find the interfaces. This is done by keeping track of the boundary $\mathcal{B}$ of each interior $\mathcal{I}$ in the dissection process.  Then instead of recursively subdividing $\mathcal{I}$, the recursion is done on $\mathcal{I}\cup \mathcal{B}$.  One level of this process is shown in \Cref{mnd}. Note how subdividing $\mathcal{I}_1 \cup \mathcal{B}$ helps identify the interfaces.  This process is defined as Modified Nested Dissection(MND) in~\cite{2019arXiv190102971C}. \Cref{mnd_multilvl} shows the application of MND to do a three level partitioning of an arbitrary graph. We refer the readers to Algorithm 2.2 of~\cite{2019arXiv190102971C} for details on the implementation of MND. Conceptually, this idea extends to hypergraph based partitioning and we adopt this in this work. 

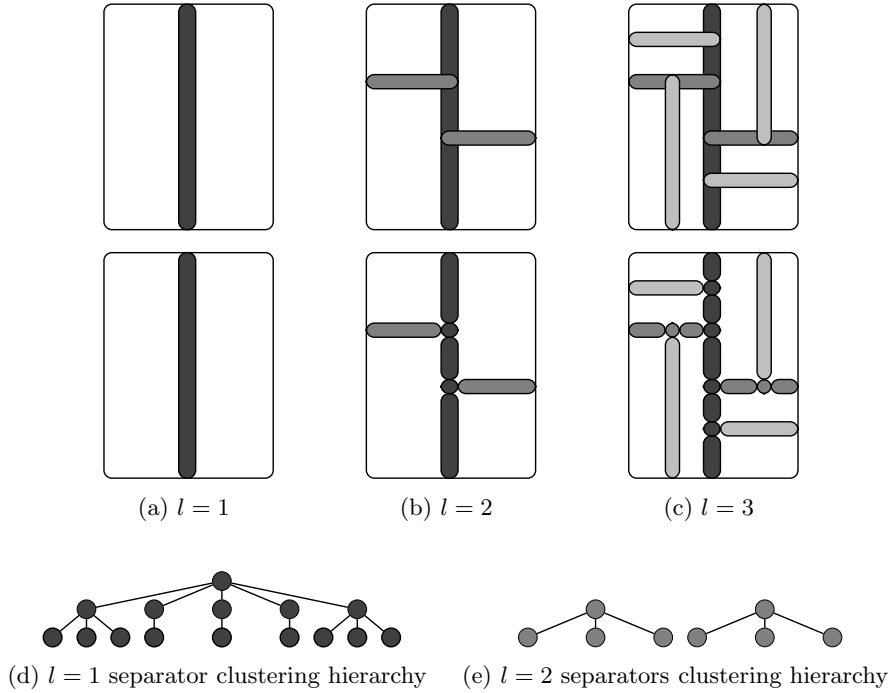
\begin{figure}[tbhp]
    \centering
    \begin{subfigure}[t]{0.25\textwidth}
        \centering
        \scalebox{0.75}{
        \begin{tikzpicture}
            \filldraw[fill = white, rounded corners, line width = 0.25mm] (0, 0) rectangle (3, 4) {};
            \filldraw[fill = darkgray, rounded corners, line width= 0.25mm] (1.325, 0) rectangle (1.625, 4) {};
        \end{tikzpicture}}
    \end{subfigure}%
   ~
    \begin{subfigure}[t]{0.25\textwidth}
        \centering
        \scalebox{0.75}{
        \begin{tikzpicture}
            \filldraw[fill = white, rounded corners, line width = 0.25mm] (0, 0) rectangle (3, 4) {};
            \filldraw[fill = darkgray, rounded corners, line width= 0.25mm] (1.325, 0) rectangle (1.625, 4) {};
            \filldraw[fill= gray, line width =0.25mm, rounded corners] (0, 2.5) rectangle (1.625, 2.75) {};
            \filldraw[fill= gray, line width =0.25mm, rounded corners] (1.325, 1.5) rectangle (3, 1.75) {};
        \end{tikzpicture}}
    \end{subfigure}%
   ~
    \begin{subfigure}[t]{0.25\textwidth}
        \centering
        \scalebox{0.75}{
        \begin{tikzpicture}
            \filldraw[fill = white, rounded corners, line width = 0.25mm] (0, 0) rectangle (3, 4) {};
            \filldraw[fill = darkgray, rounded corners, line width= 0.25mm] (1.325, 0) rectangle (1.625, 4) {};
            \filldraw[fill= gray, line width =0.25mm, rounded corners] (0, 2.5) rectangle (1.625, 2.75) {};
            \filldraw[fill= gray, line width =0.25mm, rounded corners] (1.325, 1.5) rectangle (3, 1.75) {};
            \filldraw[fill= lightgray, line width =0.25mm, rounded corners] (0.65, 0) rectangle (0.9, 2.75) {};
            \filldraw[fill= lightgray, line width =0.25mm, rounded corners] (0, 3.25) rectangle (1.625, 3.5) {};
            \filldraw[fill= lightgray, line width =0.25mm, rounded corners] (2.275, 1.5) rectangle (2.525, 4) {};
            \filldraw[fill= lightgray, line width =0.25mm, rounded corners] (1.325, 0.75) rectangle (3, 1) {};
        \end{tikzpicture}}
    \end{subfigure}%
    \vspace{0.25cm}
    
    \begin{subfigure}[t]{0.25\textwidth}
        \centering
        \scalebox{0.75}{
        \begin{tikzpicture}
            \filldraw[fill = white, rounded corners, line width = 0.25mm] (0, 0) rectangle (3, 4) {};
            \filldraw[fill = darkgray, rounded corners, line width= 0.25mm] (1.325, 0) rectangle (1.625, 4) {};
        \end{tikzpicture}}
        \caption{$l=1$}
    \end{subfigure}%
   ~
    \begin{subfigure}[t]{0.25\textwidth}
        \centering
        \scalebox{0.75}{
        \begin{tikzpicture}
            \filldraw[fill = white, rounded corners, line width = 0.25mm] (0, 0) rectangle (3, 4) {};
            \filldraw[fill = darkgray, rounded corners, line width= 0.25mm] (1.325, 0) rectangle (1.625, 1.5) {};
            \filldraw[fill = darkgray, rounded corners, line width= 0.25mm] (1.325, 1.5) rectangle (1.625, 1.75) {};
            \filldraw[fill = darkgray, rounded corners, line width= 0.25mm] (1.325, 1.75) rectangle (1.625, 2.5) {};
            \filldraw[fill = darkgray, rounded corners, line width= 0.25mm] (1.325, 2.5) rectangle (1.625, 2.75) {};
            \filldraw[fill = darkgray, rounded corners, line width= 0.25mm] (1.325, 2.75) rectangle (1.625, 4) {};
            \filldraw[fill= gray, line width =0.25mm, rounded corners] (0, 2.5) rectangle (1.325, 2.75) {};
            \filldraw[fill= gray, line width =0.25mm, rounded corners] (1.625, 1.5) rectangle (3, 1.75) {};
        \end{tikzpicture}}
        \caption{$l=2$}
    \end{subfigure}%
   ~
    \begin{subfigure}[t]{0.25\textwidth}
        \centering
        \scalebox{0.75}{
        \begin{tikzpicture}
            \filldraw[fill = white, rounded corners, line width = 0.25mm] (0, 0) rectangle (3, 4) {};
           \filldraw[fill = darkgray, rounded corners, line width= 0.25mm] (1.325, 0) rectangle (1.625, 0.75) {};
           \filldraw[fill = darkgray, rounded corners, line width= 0.25mm] (1.325, 0.75) rectangle (1.625, 1) {};
           \filldraw[fill = darkgray, rounded corners, line width= 0.25mm] (1.325, 1) rectangle (1.625, 1.5) {};
            \filldraw[fill = darkgray, rounded corners, line width= 0.25mm] (1.325, 1.5) rectangle (1.625, 1.75) {};
            \filldraw[fill = darkgray, rounded corners, line width= 0.25mm] (1.325, 1.75) rectangle (1.625, 2.5) {};
            \filldraw[fill = darkgray, rounded corners, line width= 0.25mm] (1.325, 2.5) rectangle (1.625, 2.75) {};
            \filldraw[fill = darkgray, rounded corners, line width= 0.25mm] (1.325, 2.75) rectangle (1.625, 3.25) {};
            \filldraw[fill = darkgray, rounded corners, line width= 0.25mm] (1.325, 3.25) rectangle (1.625, 3.5) {};
            \filldraw[fill = darkgray, rounded corners, line width= 0.25mm] (1.325, 3.5) rectangle (1.625, 4) {};
            \filldraw[fill= gray, line width =0.25mm, rounded corners] (0, 2.5) rectangle (0.65, 2.75) {};
            \filldraw[fill= gray, line width =0.25mm, rounded corners] (0.65, 2.5) rectangle (0.9, 2.75) {};
            \filldraw[fill= gray, line width =0.25mm, rounded corners] (0.9, 2.5) rectangle (1.325, 2.75) {};
            \filldraw[fill= gray, line width =0.25mm, rounded corners] (1.625, 1.5) rectangle (2.275, 1.75) {};
            \filldraw[fill= gray, line width =0.25mm, rounded corners] (2.275, 1.5) rectangle (2.525, 1.75) {};
            \filldraw[fill= gray, line width =0.25mm, rounded corners] (2.525, 1.5) rectangle (3, 1.75) {};
            \filldraw[fill= lightgray, line width =0.25mm, rounded corners] (0.65, 0) rectangle (0.9, 2.5) {};
            \filldraw[fill= lightgray, line width =0.25mm, rounded corners] (0, 3.25) rectangle (1.325, 3.5) {};
            \filldraw[fill= lightgray, line width =0.25mm, rounded corners] (2.275, 1.75) rectangle (2.525, 4) {};
            \filldraw[fill= lightgray, line width =0.25mm, rounded corners] (1.625, 0.75) rectangle (3, 1) {};
        \end{tikzpicture}}
        \caption{$l=3$}
    \end{subfigure}
    
    \vspace{0.25cm}
    
    \begin{subfigure}[t]{0.45\textwidth}
        \centering
        \scalebox{0.75}{
        \begin{tikzpicture}[level distance=0.5cm,
  level 1/.style={sibling distance=1.2cm},
  level 2/.style={sibling distance=0.6cm},
  edge from parent/.append style = {line width = 0.25mm}]
  \node (Root) [circle, draw=black, fill=darkgray] at (0,0) {}
    child {node [circle, draw=black, fill=darkgray]{}
        child {node [circle, draw=black, fill=darkgray]{}}
        child {node [circle, draw=black, fill=darkgray]{}}
        child {node [circle, draw=black, fill=darkgray]{}}
    }
    child {node [circle, draw=black, fill=darkgray]{}
        child {node [circle, draw=black, fill=darkgray]{}}
        }
    child {node [circle, draw=black, fill=darkgray]{} 
        child {node [circle, draw=black, fill=darkgray]{}}
        }
    child {node [circle, draw=black, fill=darkgray]{} 
        child {node [circle, draw=black, fill=darkgray]{}}
        }
    child {node [circle, draw=black, fill=darkgray]{}
        child {node [circle, draw=black, fill=darkgray]{}}
        child {node [circle, draw=black, fill=darkgray]{}}
        child {node [circle, draw=black, fill=darkgray]{}}
    };
    \end{tikzpicture}}
        \caption{$l=1$ separator clustering hierarchy}
        \label{lvl1_ch}
    \end{subfigure}%
   ~
    \begin{subfigure}[t]{0.45\textwidth}
        \centering
        \scalebox{0.75}{
        \begin{tikzpicture}[level distance=0.5cm,
  level 1/.style={sibling distance=1.2cm},
  edge from parent/.append style = {line width = 0.25mm}]
  \node (Root) [circle, draw=black, fill=gray] at (0, -0.5){}
        child {node [circle, draw=black, fill=gray]{}}
        child {node [circle, draw=black, fill=gray]{}}
        child {node [circle, draw=black, fill=gray]{}};
    \node (Root2) [circle, draw=black, fill=gray] at (3, -0.5){}
        child {node [circle, draw=black, fill=gray]{}}
        child {node [circle, draw=black, fill=gray]{}}
        child {node [circle, draw=black, fill=gray]{}};
    \end{tikzpicture}}
        \caption{$l=2$ separators clustering hierarchy}
        \label{lvl2_ch}
    \end{subfigure}
    \caption{The first row depicts the creation of separators by recursive application of modified nested dissection. The second row shows the creation of interfaces in each separator. The last row shows the clustering hierarchy within each separator.}
    \label{mnd_multilvl}
\end{figure}
Modified Nested dissection on $A^TA$ or modified HUND on $A$ defines the separators/interfaces. The columns of the matrix are reordered following the ND/HUND ordering. The rows of the matrix are reordered after column ordering and clustering is done. Row ordering has to be done such that the off-diagonal blocks are low rank and the diagonal blocks are full rank. 

We employ a different heuristics to assign the rows to the clusters. For diagonally dominant matrices, the reordering of the rows can be the same as the columns. For general matrices, one heuristic is to identify the cluster such that the weight of the row in that cluster is maximized. In other words, row $r_i$ is assigned to cluster $c$ where $c = \arg \max_{c_k}\sum_{j \in c_k} A_{ij}^2$. However, this can lead to too many rows assigned to a single cluster resulting in rectangular diagonal blocks. Typically, we want to avoid this situation as we want all the diagonal blocks to be full rank.

Another heuristic is to permute large entries to the diagonal of the matrix. This is done by performing a bipartite matching between the rows and the columns of the matrix. We use the MC64 routine from the HSL Mathematical Software Library~\cite{hsl_mc64} to perform the matching. One can test the performance with different heuristics and choose the best one for their problem. 

\subsection{Householder QR on Separators}
\label{sparseQR_S}

The factorization of interiors or separators at a level $l$ is done by applying a block Householder step (regular sparse QR). Here, we describe the QR factorization of a separator $s$  reinterpreted in our notation. Let $s$ be the separator of interest, $n$ be all its neighbors (i.e, $(A^TA)_{ns} \ne 0)$ and $w$ be the rest of the nodes disconnected from $s$ in the graph of $A^TA$. Let nodes in $n$ be further categorized into $n=\{n_1, n_2, n_3 \}$. Nodes $n_1$ are such that $A_{n_1s} \ne 0$, while $A_{sn_1}$ may or may not be zero. Nodes $n_2$ are such that $A_{n_2s} = 0$ and $A_{sn_2} \ne 0$ and nodes $n_3$ are such that $A_{n_1n_3} \ne 0$, $A_{sn_3}=0$ and $A_{n_3s}=0$. All such nodes $n$ will correspond to $(A^TA)_{ns} \ne 0$.  Consider the matrix A blocked in the following form, 
\[ A = \begin{bmatrix}
A_{ss} & A_{sn_1} & A_{sn_2} & & \\
A_{n_1s} & A_{n_1n_1} & & A_{n_1n_3} & \\
 & A_{n_2n_1} &  A_{n_2n_2} & A_{n_2n_3} & A_{n_2w} \\
  & A_{n_3n_1} &  A_{n_3n_2} & A_{n_3n_3} & A_{n_3w} \\
& A_{wn_1} & A_{wn_2} & A_{wn_3} & A_{ww}
\end{bmatrix} \]

All the diagonal blocks are square as explained in the previous section. Consider the block Householder matrix $H$ such that, 
\[ H^T \begin{bmatrix}
A_{ss} \\
A_{n_1s}
\end{bmatrix}  = \begin{bmatrix}
R_{ss} \\
\\
\end{bmatrix}\]
where $R_{ss} \in \mathbb{R}^{|s|\times |s|}$ is upper triangular. Define, 
\[H_s = 
  \begin{bmatrix}
H & \\
& I
\end{bmatrix}\] Then, \[ H_s^{T}A = \begin{bmatrix}
R_{ss} & R_{sn_1} & R_{sn_2} & R_{sn_3} & \\
& \Tilde{A}_{n_1n_1} & \Tilde{A}_{n_1n_2}& \Tilde{A}_{n_1n_3} & \\
& A_{n_2n_1} &  A_{n_2n_2} & A_{n_2n_3} & A_{n_2w} \\
  & A_{n_3n_1} &  A_{n_3n_2} & A_{n_3n_3} & A_{n_3w} \\
& A_{wn_1} & A_{wn_2} & A_{wn_3} & A_{ww}
\end{bmatrix} = \begin{bmatrix}
R_{ss} & R_{sn} & \\
& \Tilde{A}_{nn}& A_{nw} \\
& A_{wn} & A_{ww}
\end{bmatrix} \]
Define, \[R_s = \begin{bmatrix}
R_{ss} & R_{sn} & \\
& I_n & \\
& & I_w
\end{bmatrix} \]
Then, 
\[H_s^{T} A R_s^{-1} = \begin{bmatrix}
I_s &  & \\
& \Tilde{A}_{nn} & {A}_{nw} \\
& A_{wn} & A_{ww}
\end{bmatrix} 
\]
Hence the cluster $s$ has been disconnected from the rest. In this process we have introduced fill-in only between the neighbors $n$. There are no additional non-zeros in the blocks involving $w$ ($A_{nw}$, $A_{wn}$, and $A_{ww}$). This is key in the ND ordering.

\subsection{Sparsification of Interfaces}
\label{spars_s}
Once the interiors/separators at a level $l$ have been factorized, the algorithm goes through each interface and sparsifies it. Consider an interface $p$, 
\[A = 
\begin{bmatrix}
A_{pp} & A_{pn} &  \\
A_{np} & A_{nn} & A_{nw} \\
& A_{wn} & A_{ww}
\end{bmatrix}\]
Assume the off-diagonal blocks $A_{np}$ and $A_{pn}$ are low-rank. Hence, the matrix $\begin{bmatrix}
A_{np}^T & \sigma A_{pp}^TA_{pn}
\end{bmatrix}$ can be well-approximated by a low rank matrix (for a scalar $\sigma$ to be defined later).
\[\begin{bmatrix}
A_{np}^{T} & \sigma A_{pp}^{T}A_{pn}
\end{bmatrix}  = Q_{pp}W_{pn} = \begin{bmatrix}
Q_{pf} & Q_{pc}
\end{bmatrix}\begin{bmatrix}
W_{fn} \\
W_{cn}
\end{bmatrix} \text{with } \|W_{fn}\|_{_2}=\mathcal{O}(\epsilon)\]
\[\begin{bmatrix}
W_{fn}\\
W_{cn}
\end{bmatrix} = \begin{bmatrix}
W_{fn}^{(1)} & W_{fn}^{(2)} \\
W_{cn}^{(1)} & W_{cn}^{(2)}
\end{bmatrix} \]
Then, 
\[A_{np}Q_{pc} = W_{cn}^{(1)T}, \text{ } A_{np}Q_{pf} = W_{fn}^{(1)T} = \mathcal{O}(\epsilon)\]
Define, \[Q_p = \begin{bmatrix}
Q_{pp} & & \\
 & I & \\
 & & I
\end{bmatrix}\]
\[ AQ_p = \begin{bmatrix}
\Tilde{A}_{ff} & \Tilde{A}_{fc} & A_{fn}  & \\
 \Tilde{A}_{cf} & \Tilde{A}_{cc} & A_{cn} &\\
 \mathcal{O}(\epsilon) & W_{cn}^{(1)T} & A_{nn} & A_{nw} \\ 
 & & A_{wn} & A_{ww}
\end{bmatrix} \quad \text{where,} \quad 
A_{pn} = \begin{bmatrix}
A_{fn} \\
A_{cn}
\end{bmatrix}\]
where $\Tilde{A}_{ff}$ is a square block of size $|f| \times |f|$. Dropping the $\mathcal{O}(\epsilon)$ and applying a block Householder $H_f$ on the $f$ block, (see \Cref{sparseQR_S}),
\[ H_f = \begin{bmatrix}
H & \\
& I
\end{bmatrix}\] where $H \in \mathbb{R}^{|p|\times |p|}$. If $H_{ff}$ represent the first $|f|$ columns of $H$, then $H_{ff}^T (AQ_p)_{(:,1:f)} = R_{ff}$
\[H_f^T A Q_p = \begin{bmatrix}
R_{ff} & R_{fc} & \r{R_{fn}} & \\
& \hat{A}_{cc} & \hat{A}_{cn} & \\
& W_{cn}^{(1)T} & A_{nn} & A_{nw} \\
& & A_{wn} & A_{ww}
\end{bmatrix} \]
The term $R_{fn}=\mathcal{O}(\epsilon)$ for an appropriate choice of the scalar $\sigma$. The value of $\sigma$ for which this is true is given by \Cref{lemma1}. The proof is given in \Cref{proof:lem1}. 
\begin{restatable}{lemma}{sigchoice} \label{lemma1}
$\|R_{fn}\|_{_2} \leq \epsilon$, for $\sigma = \frac{1}{\sigma_{\text{min}}(A_p)}$ where $A_p = \begin{bmatrix}
A_{pp} \\
A_{np}
\end{bmatrix}$ 
\end{restatable}
Finally define, \[R_f = \begin{bmatrix}
R_{ff} & R_{fc} &  & \\
& I_c & & \\
& & I_n & \\
& & & I_w
\end{bmatrix}\]
to get,
\[H_f^T AQ_p R_f^{-1} = \begin{bmatrix}
I_f & & & \\
& \hat{A}_{cc} & \hat{A}_{cn} & \\
& W_{cn}^{(1)T} & A_{nn} & A_{nw} \\
& & A_{wn} & A_{ww}
\end{bmatrix}\]
Hence, the fine nodes $f$ are disconnected from all the remaining nodes. The size of interface $p$ is decreased by $|f|$. The $A_{nn}$, $A_{nw}$, $A_{wn}$, and $A_{ww}$ blocks are not affected during the sparsification process. Thus, we could eliminate a part of $p$ without introducing additional nonzeros in the rest of the matrix. Note that, the last two statements are true even if the term $R_{fn}$ was not $\mathcal{O}(\epsilon)$. 

However, it is important that $\|R_{fn}\|_{_2}\leq \epsilon$ to ensure that the elimination tree structure of $A^TA$ is not affected. Remember that the QR factorization on $A$ and Cholesky on $A^TA$ are directly related. Hence, we need to ensure that we have not introduced fill-in in the $n-n$, $n-w$, $w-w$ blocks of $A^TA$ as well.

To understand this better, consider two nodes $n_1$ and $n_2$ such that $n_1, n_2\in  n$ and belong to two disjoint subtrees of the elimination tree (of $A^TA$).  Then by definition, (see Corollary 3.2 in~\cite{elimination_tree}) $R_{n_1n_2} = 0$ during direct QR factorization on $A$.  However, say that $(A^TA)_{n_1n_2} \neq 0$ after sparsification of an interface in spaQR. This implies that an Householder transformation on the column $A_{:,n_1}$ will modify the column $A_{:,n_2}$, since the columns are not orthogonal \big($(A^TA)_{n_1n_2} \neq 0$\big). Ignoring any spurious cancellations that can occur, this leads to $R_{n_1n_2}\neq 0$. Thus, the fill-in guarantees that come with following the elimination tree ordering of the unknowns do not hold anymore. 

In \Cref{well_sep}, we show that sparsification does not affect the elimination tree of $A^TA$, that  is, any two disjoint subtrees of the elimination tree remain disjoint after sparsification of any interface. The proof depends on \Cref{lemma1} and is given in \Cref{proof:wellsep}.

\begin{restatable}{theorem}{wellsep}
\label{well_sep}
For any two interfaces $l$, $m$ such that the block  $R_{lm} = 0$ in the direct QR factorization, we have $R_{lm} \approx 0$ in spaQR as well. 
\end{restatable}

\subsection{Scaling of Interfaces}
\label{scale_s}

The $\sigma$ factor in the sparsification step was chosen to be $\sigma_{\text{min}}(A_p)^{-1}$. This factor was necessary to ensure that $R_{fn} = \mathcal{O}(\epsilon)$ in \Cref{lemma1}, which in turn was necessary to prove \Cref{well_sep}. However when $A_p$ (or $A$) is ill-conditioned, $\sigma$ can be large which will lead to a slower decay of the singular values of $\begin{bmatrix}
A_{np}^T & \sigma A_{pp}^TA_{pn}
\end{bmatrix}$.  Thus even if the off-diagonal blocks have a faster decay of singular values, we could not take full advantage of it. In addition to fixing this, we get improved accuracy by scaling the diagonal blocks corresponding to all interfaces before sparsification. This gives better error guarantees as shown in \Cref{scaling_err_s}. Similar rescaling ideas have been shown to improve accuracy in~\cite{2019arXiv190102971C,FeliuFab2018RecursivelyPH, Xia2017EffectiveAR} for sparse Cholesky factorization on hierarchical matrices. 

Consider an interface $p$ and its neighbors $n$,
\[A = \begin{bmatrix}
A_{pp} & A_{pn} \\
A_{np} & A_{nn}
\end{bmatrix}\]
Find the QR decomposition of $A_{pp}$; $A_{pp} = U_{pp} R_{pp}$. Then \[U_{pp}^T A_{pp} R_{pp}^{-1} = I\]
  Define, \[U_p = \begin{bmatrix}
 U_{pp}^T  &\\
 & I
 \end{bmatrix} \qquad R_p = \begin{bmatrix}
 R_{pp}^{-1} & \\
 & I_n
 \end{bmatrix}\]
 Then, 
 \[U_p^TAR_p = \begin{bmatrix}
 I_p & \Tilde{A}_{pn} \\
 \Tilde{A}_{np} & A_{nn}
 \end{bmatrix} \]
Similarly we scale the diagonal blocks corresponding to all the remaining interfaces. Once the interfaces are scaled, sparsification is straightforward; compress, \[ \begin{bmatrix}
 \Tilde{A}_{np}^T & \Tilde{A}_{pn}
 \end{bmatrix} = Q_{pp}W_{pn} = \begin{bmatrix}
Q_{pf} & Q_{pc}
\end{bmatrix}\begin{bmatrix}
W_{fn} \\
W_{cn}
\end{bmatrix} \quad \text{with} \quad \|W_{fn}\|_{_2}=\mathcal{O}(\epsilon) \]
Defining $Q_p$ as in \Cref{spars_s}, we find that sparsification and factorization of the `fine' nodes boils down to applying $Q_p$ on the left and right of the matrix. 
\[Q_p^T U_p^TAR_p Q_p = \begin{bmatrix}
 I_f & & \r{E_2} \\
 & I_c & \hat{A}_{cn}\\
 \r{E_1} & \hat{A}_{nc} & A_{nn}
\end{bmatrix}\]
where $\r{E_1} = W_{fn}^{(1)T}$, $\r{E_2} = W_{fn}^{(2)}$ and $\|E_1\|_{_2}\approx\|E_2\|_{_2}\leq \epsilon$. Since \Cref{lemma1} holds true, \Cref{well_sep} also holds. Hence, the algorithm can proceed without breaking the elimination tree structure.

\subsection{Merging of clusters}
\label{sec:merge}
Once the factorization of separators at a level is done, the interfaces of the remaining ND separators are merged following the cluster hierarchy. For example, in \Cref{mnd_multilvl}, once the leaves $l=4$ and the $l=3$ separators are factorized, the interfaces of the separators at $l=1,2$ are merged following the clustering hierarchy shown in \Cref{lvl1_ch}, \Cref{lvl2_ch}. Merging simply means combining the block rows and columns of the interfaces into a single block matrix. 

\subsection{Sparsified QR}

We now have all the building blocks to write down the spaQR algorithm. Given a matrix, we typically pre-process it so that the 2-norm of each column is a constant. Then the matrix is partitioned to identify separators, interfaces (\Cref{ord_clus}) and is appropriately reordered. The spaQR algorithm involves applying a sequence of block Householder factorizations $H_s, R_s$ (\Cref{sparseQR_S}), scaling $U_p, R_p$ (\Cref{scale_s}), sparsification of the interfaces $Q_p$ (\Cref{spars_s}), permutations to take care of the fine nodes and merging of the clusters (\Cref{sec:merge}), at each level $l$ such that, 
\[ Q^T A W^{-1} \approx I\] 
where,
\begin{align*}
    Q &= \prod_{l=1}^{L}\Bigg( \prod_{s\in S_l}H_s \prod_{p\in C_l}U_p \prod_{p\in C_l} Q_p\Bigg)\\
    W &= \prod_{l=L}^{1}\Bigg(\prod_{p\in C_l} Q_p^T\prod_{p\in C_l}R_p \prod_{s\in S_l}R_s  \Bigg)
\end{align*}
\begin{algorithm}
\caption{Sparsified QR (spaQR) algorithm}
\begin{algorithmic}[1]
  \REQUIRE {Sparse matrix A, Tolerance $\epsilon$}
   \STATE {Compute column and row partitioning of A, infer separators and interfaces (see \Cref{ord_clus})}
   \FORALL{$l=L, L-1, \dots 1$}
   \FORALL{separators $s$ at level $l$}
        \STATE {Factorize $s$ using block Householder (see \Cref{sparseQR_S})}
        \STATE {Append $H_s$ to $Q$ and $R_s$ to $W$}
   \ENDFOR
   \FORALL{interfaces $p$ remaining at level $l$}
        \STATE {Perform block diagonal scaling on  $p$ (see \Cref{scale_s})}
        \STATE{Append $U_p$ to $Q$ and $R_p$ to $W$}
   \ENDFOR
   \FORALL{interfaces $p$ remaining at level $l$}
        \STATE {Sparsify interface $p$ (see \Cref{spars_s}, \Cref{scale_s})}
        \STATE{Append $Q_p$ to $Q$ and $Q_p^T$ to $W$}
   \ENDFOR
   \FORALL{separators $s$ remaining at level $l$}
   \STATE {Merge interfaces of $s$ one level following the cluster hierarchy (see \Cref{sec:merge})}
   \ENDFOR
   \ENDFOR
  \RETURN {$Q = \prod_{l=1}^{L}\Bigg( \prod_{s\in S_l}H_s \prod_{p\in C_l}U_p \prod_{p\in C_l} Q_p\Bigg)$\\ 
  \qquad \qquad $W = \prod_{l=L}^{1}\Bigg(\prod_{p\in C_l} Q_p^T\prod_{p\in C_l}R_p \prod_{s\in S_l}R_s  \Bigg)$ such that $Q^TAW^{-1} \approx I$ }
  
\end{algorithmic}
\label{Algo: spaQR}
\end{algorithm}

Here, $S_l$ is the set of all separators at level $l$ in the elimination tree and $C_l$ is the set of all interfaces remaining after factorization of separators at level $l$. $Q$ is a product of orthogonal matrices and $W$ is a product of upper triangular and orthogonal matrices. Since, $Q$ and $W$ are available as sequence of elementary transformations, they are easy to invert. The complete algorithm is presented in \Cref{Algo: spaQR}.

\section{Theoretical results}
\label{theoretical_results_sec}

In this section, we  study the error introduced during the sparsification process, the effect of scaling and the effectiveness of using spaQR as a preconditioner with iterative methods. Finally, we discuss the theoretical complexity of the spaQR algorithm. 

\subsection{Error Analysis}
\label{Error_s}

Consider a simple $2\times 2$ block matrix A. 
\[A = \begin{bmatrix}
A_{pp} & A_{pn} \\
A_{np} & A_{nn}
\end{bmatrix}\]
After sparsification, interface $p$ is split into fine $f$ and coarse $c$ nodes,
\[AQ_p = 
\begin{bmatrix}
A_{ff} & A_{fc} & A_{fn}  \\
A_{cf} & A_{cc} & A_{cn} \\
\r{E} & A_{nc} & A_{nn}
\end{bmatrix}\]
where $\|\r{E}\|_{_2}\leq \epsilon$. After performing Householder QR on the $f$ columns,
\begin{align*}
  H_f^TA Q_p &= \begin{bmatrix}
R_{ff} & R_{fc} & \r{R_{fn}} \\
& \hat{A}_{cc} & \hat{A}_{cn} \\
\r{E} & A_{nc} & A_{nn}
\end{bmatrix} \\
&= \begin{bmatrix}
I_f &  & \r{R_{fn}} \\
& \hat{A}_{cc} & \hat{A}_{cn} \\
\r{E}R_{ff}^{-1} & A_{nc}-\r{E}R_{ff}^{-1}R_{fc} & A_{nn}
\end{bmatrix} \begin{bmatrix}
R_{ff} & R_{fc} & \\
& I_c &  \\
& & I_n
\end{bmatrix} 
\end{align*}
where $\|\r{R_{fn}}\|_{_2}\leq \epsilon$. Then, 
\[H_f^TAQ_pR_f^{-1} = \begin{bmatrix}
I_f &  & \r{R_{fn}} \\
& \hat{A}_{cc} & \hat{A}_{cn} \\
\r{E}R_{ff}^{-1} & A_{nc}-\r{E}R_{ff}^{-1}R_{fc} & A_{nn}
\end{bmatrix}\]
Define, 
\[H_f^T\Tilde{A}Q_pR_f^{-1} = \begin{bmatrix}
I_f &  &  \\
& \hat{A}_{cc} & \hat{A}_{cn} \\
 & A_{nc} & A_{nn}
\end{bmatrix}\]
as the approximation when $\r{E}$ and $\r{R_{fn}}$ are dropped in our algorithm. Then the error in the approximation is, 
\[H_f^T(A-\Tilde{A})Q_pR_f^{-1} =\begin{bmatrix}
  &  & \r{R_{fn}}   \\
&  & \\
\r{E}R_{ff}^{-1} & -\r{E}R_{ff}^{-1}R_{fc} &
\end{bmatrix} \]
\begin{align*}
    \|H_f^T(A-\Tilde{A})Q_pR_f^{-1}\|_{_2} & \leq c_1 \|\r{E}R_{ff}^{-1}R_{fc}\|_{_2}
    \leq c_1 \|E\|_{_2} \; \|R_{ff}^{-1}\|_{_2} \; \|R_{fc}\|_{_2} \\
    &\leq c_1\epsilon \; \frac{1}{\sigma_{\text{min}}(A_p)} \; \sigma_{\text{max}}(A_p)
    = c_1\kappa(A_p) \; \epsilon
\end{align*}
where $c_1$ is a constant. We have used that facts that, 
$$\begin{bmatrix}
R_{fc} \\
R_{cc} \\
A_{nc}
\end{bmatrix} = H_f^T \begin{bmatrix}
A_{fc} \\
A_{cc} \\
A_{nc}
\end{bmatrix}
\quad \text{and} \quad
\|R_{fc}\|_{_2} \leq \Bigg\|\begin{bmatrix}
A_{fc} \\
A_{cc} \\
A_{nc}
\end{bmatrix}\Bigg\|_{_2} \leq \Bigg\|\begin{bmatrix}
A_{pp} \\
A_{np} \\
\end{bmatrix}\Bigg\|_{_2} 
= \sigma_{\text{max}}(A_p)
$$
in proving the above result. Thus, when $A_p$ is ill-conditioned, it is possible that $R_{ff}$ is ill-conditioned and the error in the approximation is worse than $\epsilon$. We can improve the upper bound on the error by first scaling the interfaces as we prove next. 

\subsection{Accuracy of scaling}
\label{scaling_err_s}

Scale the diagonal blocks of all interfaces before sparsification as outlined in \Cref{scale_s}. If $U$ is the scaled version of $A$, then $H_f = Q_p$ and $R_f = I$. Then, 
\[ Q_p^TUQ_p = \begin{bmatrix}
 I_f & &  \\
 & I_c & \hat{A}_{cn}\\
  & \hat{A}_{nc} & I_n
 \end{bmatrix} +  \begin{bmatrix}
  & & \r{E_2} \\
 & & \\
 \r{E_1} &  & 
 \end{bmatrix} \]
 Define, 
 \[ Q_p^T \Tilde{U} Q_p = \begin{bmatrix}
 I_f & &  \\
 & I_c & \hat{A}_{cn}\\
  & \hat{A}_{nc} & I_n
 \end{bmatrix} \]
 Then the approximation error is,
\[
    \|Q_p^T(U-\Tilde{U})Q_p\|_{_2} = \|E_1\|_{_2} = \|E_2\|_{_2} \leq \epsilon
\]
Thus, we have a better error bound by rescaling the diagonal blocks before sparsification. 

\subsection{Effectiveness of the preconditioner}
Consider the same $2 \times 2$ block matrix A. After scaling and sparsification of interface $p$, we have 
\[
Q_p^TU Q_p =\begin{bmatrix}
 I_f & &  \\
 & I_c & \hat{A}_{cn}\\
  & \hat{A}_{nc} & I_n
 \end{bmatrix} +  \begin{bmatrix}
  & & \r{E_2} \\
  &  & \\
 \r{E_1} &  & 
 \end{bmatrix}
 \]
 Let us complete the factorization by  performing an exact QR factorization on the $c$ and $n$ blocks as follows
 \begin{align*}
 H_c^T Q_p^T U Q_p &=\begin{bmatrix}
 I_f & &  \\
 & R_{cc} & R_{cn}\\
  &  & \hat{A}_{nn}
 \end{bmatrix} + H_c^T \begin{bmatrix}
  & & \r{E_2} \\
 & & \\
 \r{\Tilde{E}_1} &  & 
 \end{bmatrix}\\
  H_c^T Q_p^TU Q_p R_c^{-1}
 &= \begin{bmatrix}
 I_f & &  \\
 & I_c & \\
  &  & \hat{A}_{nn}
 \end{bmatrix} + H_c^T \begin{bmatrix}
  & & \r{E_2} \\
  & & \\
 \r{\Tilde{E}_1} &  & 
 \end{bmatrix} \\
  S = H_n^TH_c^T Q_p^TU Q_p R_c^{-1} R_n^{-1} &=\begin{bmatrix}
 I_f & &  \\
 & I_c & \\
  &  & I_n
 \end{bmatrix} + H_n^TH_c^T \begin{bmatrix}
  & & \r{E_2}R_{nn}^{-1} \\
 & & \\
 \r{\Tilde{E}_1} &  &
 \end{bmatrix}
\end{align*}
With this, we have $S$ as the preconditioned matrix. The final error is, 
\[E = H_n^TH_c^T \begin{bmatrix}
  & & \r{E_2}R_{nn}^{-1} \\
 & & \\
 \r{\Tilde{E}_1} &  &
 \end{bmatrix}
\]
If we represent $H_c = \begin{bmatrix}
H_{cc} & H_{cn}
\end{bmatrix}$, then, $\hat{A}_{nn} = H_{cn}^T \begin{bmatrix}
\hat{A}_{cn} \\
I_n
\end{bmatrix}$. Since, $\hat{A}_{nn}$ is a product of an orthogonal and a well-conditioned matrix, $\hat{A}_{nn}$ is also well-conditioned.  Therefore, $\|R_{nn}^{-1}\|_{_2} = \mathcal{O}(1)$.
Then, 
\[ \|E\|_{_2} = \mathcal{O}(\epsilon)
\]
The condition number of the preconditioned matrix $S = I+E$ can be calculated as follows,
\[ \sigma_{\max} (S) = \max_{x\in \mathbb{R}^{M}} \frac{\|Ix+Ex\|_{_2}}{\|x\|_{_2}} \leq \max_{x\in \mathbb{R}^{M}} \frac{\|Ix\|_{_2}}{\|x\|_{_2}} + \max_{x\in  \mathbb{R}^{M}} \frac{\|Ex\|_{_2}}{\|x\|_{_2}} = 1+\|E\|_{_2} \]
\[ \sigma_{\min} (S) = \min_{x\in \mathbb{R}^{M}} \frac{\|Ix+Ex\|_{_2}}{\|x\|_{_2}} \geq \min_{x\in \mathbb{R}^{M}} \frac{\|Ix\|_{_2}}{\|x\|_{_2}} - \max_{x\in \mathbb{R}^{M}} \frac{\|Ex\|_{_2}}{\|x\|_{_2}} = 1-\|E\|_{_2} \]
Therefore, 
\[ \kappa(S) \leq \frac{1+ \|E\|_{_2}}{1-\|E\|_{_2}}\]

\subsection{Complexity Analysis}
\label{sec:complexity}

In this section, we discuss the complexity of the spaQR algorithm under some assumptions. Consider the Nested Dissection process on the graph of $A^TA$ ($G_{A^TA}$). Define a node as a subgraph of $G_{A^TA}$. The root of the tree corresponds to $l=1$ and the root node is the entire graph $G_{A^TA}$. The children nodes are subgraphs of $G_{A^TA}$ disconnected by a separator. 

We assume that the matrices and their graphs satisfy the following properties.
\begin{enumerate}
    \item The leaf nodes in the elimination tree contain at most $N_0$ nodes, where $N_0 \in \mathcal{O}(1)$.
    \item Let $D_i$ be the set of all nodes $j$ that are descendants of a node $i$, whose size is at least $n_i/2$. We assume that the size of $D_i$ is bounded, that is, $|D_i|=\mathcal{O}(1)$ for all $i$.
    \item All the Nested Dissection separators are minimal. That is, every vertex in the separator connects two disconnected nodes in $G_{A^TA}$.
    \item The number of edges leaving a node (subgraph) of size $n_i$ is at most $n_i^{2/3}$. In other words, a node of size $n_i$ is connected to at most $n_i^{2/3}$  vertices in $G_{A^TA}$. Most matrices that arise in the discretization of 2D and 3D PDEs satisfy this property.
\end{enumerate}  

\paragraph{Direct Householder QR} We first recover the cost of direct QR on $A$ with Nested Dissection partitioning on PDEs discretized on a 3D grid. Consider a node $i$ of size $2^{-l+1}N \leq n_i \leq 2^{-l+2}N$ at a level $l$ in the elimination tree. By assumption 4, the associated separator has size at most \[c_l \in \mathcal{O}\Big(2^{-2l/3}N^{2/3}\Big)\]
The fill-in from Householder QR on the interiors results in at most $\mathcal{O}(2^{-2l/3}N^{2/3})$ non-zeros per row and column. This is because of assumption 4 and the fact that new connections are introduced only between the distance 1 neighbors of a node in $G_{A^TA}$. Thus, the cost of Householder QR on a separator is 
\[
    h_l \in \mathcal{O}\Big(\big(2^{-2l/3}N^{2/3}\big)^3\Big) 
    = \mathcal{O}\big(2^{-2l}N^2\big)
\]
By the pigeonhole principle, the number of nodes of size $n_i$, with $2^{-l+1}N \leq n_i \leq 2^{-l+2}N$ is bounded by $2^{l-1}$. Then, the total cost of a direct Householder QR on the matrix is,
\[t_{\text{QR, fact}} \in \mathcal{O}\Bigg(\sum_{l=1}^{L}2^{l}h_l\Bigg) = \mathcal{O}\Bigg(\sum_{l=1}^{L}2^{-l}N^2\Bigg) = \mathcal{O}\big(N^2\big) \qquad L \in \Theta(\log(N/N_0))\]

The cost of applying the factorization can be derived similarly. Solving with a given right-hand side $b$ involves applying a sequence of orthogonal and upper triangular transformations corresponding to the factorization of each interior/separator. Since, for a node of size $2^{-l+1}N \leq n_i \leq 2^{-l+2}N$, the associated separator has a size of $c_l$ with at most $\mathcal{O}(2^{-2l/3}N^{2/3})$ non-zeros per row/column, the total cost of applying the factorization is, 
\[t_{\text{QR, apply}}  \in \mathcal{O}\Bigg(\sum_{l=1}^{L}2^{l} \Big(2^{-2l/3}N^{2/3}\Big)^2 \Bigg) = \mathcal{O}\big(N^{4/3}\big)\]

\paragraph{spaQR} Next, we show that the complexity of spaQR factorization is $\mathcal{O}(N\log N)$. To show this, we need additional assumptions on the sparsification process and the size of interfaces defined in \Cref{ord_clus}. Remember that an interface is a multilevel partitioning of a separator constructed such that its size is comparable to the diameter of the subdomains at that level (see \Cref{int}). Assume that sparsification reduces the size of an interface at level $l$ to, 
\[c_l' \in \mathcal{O}(2^{-l/3}N^{1/3})\]
Thus the size of a separator decreases from $c_l$ to $c_l'$  before it is factorized. This means that the rank scales roughly as the diameter of the separator. This assumption is a consequence of low rank interactions between separators that are far away in $G_{A^TA}$.  This is comparable to complexity assumptions in the fast multipole method~\cite{FMM_1, greengard_rokhlin_1997}, spaND~\cite{2019arXiv190102971C}, and HIF~\cite{Ho2016HierarchicalIF}. Further, assume that an interface has  $\mathcal{O}(1)$ neighbor interfaces.

The fill-in in the sparsified QR process results in at most $\mathcal{O}(2^{-l/3}N^{1/3})$ entries in each row and column. This is in part due to the assumption on the size of the interfaces, the number of neighbor interfaces and the fact that new connections are only made between distance 1 neighbors of a node in $G_{A^TA}$. 

The total cost of spaQR factorization can be split into two parts:
\begin{itemize}
    \item Householder QR on interiors/separators. The size of a separator is $c_l' \in \mathcal{O}(2^{-l/3}N^{1/3})$ right before it is factorized and has at most $\mathcal{O}(2^{-l/3}N^{1/3})$ non-zeros per row/column. Then the cost of Householder QR on a separator is 
    \[h_l' \in \mathcal{O}\Big(\big(2^{-l/3}N^{1/3}\big)^3\Big) = \mathcal{O}\big(2^{-l}N\big)\]
    \item Scaling and sparsification of interfaces. The cost of scaling (QR on a block of size $c_l'\times c_l'$) an interface is $\mathcal{O}\big(2^{-l}N\big)$. Similarly, the cost of sparsifying (rank-revealing QR) an interface is also $\mathcal{O}\big(2^{-l}N\big)$ because of the assumptions on the size and number of non-zeros per row/column of an interface.
\end{itemize}
Hence, the total cost of the spaQR algorithm is
\[t_{\text{spaQR}} \in \mathcal{O}\Bigg(\sum_{l=1}^{L}2^{l}2^{-l}N\Bigg) = \mathcal{O}\Bigg(\sum_{l=1}^{L} N\Bigg) = \mathcal{O}(N \log N), \qquad L \in \Theta(\log(N/N_0))\]
The total cost of applying the factorization is
\[t_{\text{spaQR, apply}}  \in \mathcal{O}\Bigg(\sum_{l=1}^{L}2^{l} \Big(2^{-l/3}N^{1/3}\Big)^2 \Bigg) = \mathcal{O}(N)\]
The memory requirements scales as the cost of applying the factorization. We show some numerical results on the size of interfaces, the number of non-zeros rows and columns per interface block and the cost of sparsification per level on a typical example in \Cref{Sec: Profiling}. These experimental results corroborate the assumptions made here. 

\section{Benchmarks}
\label{benchmarks}

In this section, we benchmark the performance of the algorithm in solving  unsymmetric system of linear equations (high and low contrast advection diffusion problems) on uniform 2D and 3D grids and sparse matrices from Suite Sparse Matrix Collection~\cite{suitesparse} and SPARSKIT collection~\cite{Boisvert1997}. We use geometric partitioning on $A^TA$ to get the separators and interfaces for the advection diffusion problem on regular grids and Hypergraph based partitioning on $A$ using PaToH~\cite{atalyrek2011PaToHT} for the non-regular problems. For a given matrix $A$ and a tolerance $\epsilon$, the spaQR algorithm (\Cref{Algo: spaQR}) is used to compute an approximate factorization which is then used as a preconditioner with a suitable iterative solver. GMRES is used as the iterative solver and the convergence criteria is set as $\|Ax-b\|_{_2}/\|b\|_{_2} \leq 10^{-12}$. 

The algorithm was written in C++.  We use GCC 8.1.0 and Intel(R) MKL 2019 for Linux for the BLAS and LAPACK operations. The number of levels in the nested dissection process is chosen as $\lceil\log (N/64)/\log 2\rceil$ for a matrix of size $N \times N$. Low rank approximations are performed using LAPACK's dlaqps routine which performs a column pivoted QR on $r$ columns. The value $r$ is chosen such that $\frac{|R_{ii}|}{|R_{11}|} \geq \epsilon$ for $1\leq i \leq r$, where $R$ is the upper triangular matrix that comes out of the column pivoted QR method. We typically begin sparsification on levels 3 or 4. 

\subsection{Impact of Scaling}

We first compare the performance of the spaQR algorithm with and without the block diagonal scaling described in \Cref{scale_s}. First, we test the performance on flow problems in regular grids and then on non-regular problems.

\subsubsection{High contrast Advection Diffusion equations in 2D} 

Consider the variable coefficient advection diffusion equation, 
\[ -\nabla\big(a(\mathbf{x}) \cdot \nabla u(\mathbf{x})\big) + q \nabla \cdot \big(b(\mathbf{x}) u(\mathbf{x})\big) = f \quad \forall \mathbf{x} \in \Omega =[0,1], \quad u|_{d\Omega}=0 \]
where $a(\mathbf{x})$, $b(\mathbf{x})$ are sufficiently regular functions. In this example, the function $a(\mathbf{x})$ is a high contrast field quantized by a parameter $\rho$. Specifically, the field is built as follows on a $n \times n$ grid:
\begin{itemize}
    \item For every grid point $(i,j)$ choose $\hat{a}_{ij}$ uniformly at random between 0 and 1
    \item Smooth $\hat{a}$ by convolving with a unit-width Gaussian
    \item Define \[a_{ij} =  \begin{cases}
    \rho & \text{if }\hat{a}_{ij} \geq 0.5 \\
    \rho^{-1} & \text{otherwise }
    \end{cases}\]
\end{itemize}
The values of $b(\mathbf{x})$ and $q$ are set to 1. The equation is discretized on a uniform 2D $n \times n$ grid. The matrices corresponding to this discretization are generated using the open source code from~\cite{leopold_matrixgen}. 

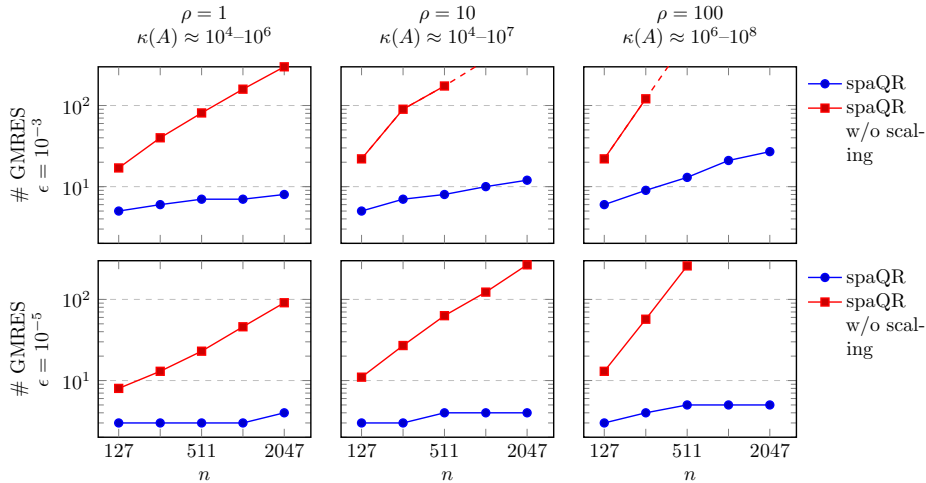
\begin{figure}[tbhp]
\centering
\begin{subfigure}{\textwidth}
\centering
\scalebox{0.75}{
\begin{tikzpicture}
\begin{loglogaxis}[
    ylabel style = {align=center},
    title style = {align = center},
    title = {$\rho = 1$ \\ $\kappa(A)\approx 10^4\text{--}10^6$},
    scale = 0.55,
    ylabel= {$\#$ GMRES \\ $\epsilon=10^{-3}$},
    xmin=90, xmax=3200,
    ymin=2, ymax=300,
    xtick = \empty,
    extra x ticks = {127, 255, 511, 1023, 2047},
    extra x tick labels = \empty,
    ymajorgrids=true,
    line width=0.25mm,
    grid style=dashed
]
\addplot coordinates {
    (127, 5)(255, 6) (511, 7) (1023, 7) (2047, 8)
    };
\addplot coordinates {
    (127, 17)(255,40) (511, 81) (1023, 159) (2047, 300)
    };
\end{loglogaxis}
\end{tikzpicture}}
\scalebox{0.75}{
\begin{tikzpicture}
\begin{loglogaxis}[
title style = {align = center},
    title = {$\rho = 10$ \\ $\kappa(A)\approx 10^4\text{--}10^7$},
    scale = 0.55,
    xmin=90, xmax=3200,
    ymin=2, ymax=300,
    xtick = \empty,
    yticklabel = \empty,
    extra x ticks = {127, 255, 511, 1023, 2047},
    extra x tick labels = \empty,
    ymajorgrids=true,
    line width=0.25mm,
    grid style=dashed
]
\addplot coordinates {
    (127, 5)(255, 7) (511, 8) (1023, 10) (2047,12)
    };
\addplot coordinates {
    (127, 22)(255,90) (511, 174) (1023,nan ) (2047, nan)
    };
\addplot[draw = red, dashed] coordinates {
    (127, 22)(255,90) (511, 174) (1023, 360 ) (2047, nan)
    };
\end{loglogaxis}
\end{tikzpicture}}
\scalebox{0.75}{
\begin{tikzpicture}
    \pgfplotsset{
every axis legend/.append style={
at={(1.02,1)},
anchor=north west,
text width=40pt},
legend cell align=left}
\begin{loglogaxis}[
title style = {align = center},
    title = {$\rho = 100$ \\ $\kappa(A)\approx 10^6\text{--}10^8$},
legend columns = 1,
legend style = {draw = none},
    scale = 0.55,
    xmin=90, xmax=3200,
    ymin=2, ymax=300,
    xtick = \empty,
    extra x ticks = {127, 255, 511, 1023, 2047},
    extra x tick labels = \empty,
    ymajorgrids=true,
    yticklabel = \empty,
    line width=0.25mm,
    grid style=dashed,
    legend entries  = {spaQR, spaQR w/o scaling}
]
\addplot coordinates {
    (127, 6)(255, 9) (511, 13) (1023, 21) (2047, 27)
    };
\addplot coordinates {
    (127, 22)(255,121) (511, nan) (1023, nan) (2047, nan)
    };
\addplot[draw = red, dashed] coordinates {
    (127, 22)(255,121) (511, 660) (1023, nan) (2047, nan)
    };
\end{loglogaxis}
\end{tikzpicture}}
\end{subfigure}

\begin{subfigure}{\textwidth}
\centering
\scalebox{0.75}{
\begin{tikzpicture}
\begin{loglogaxis}[
    ylabel style = {align=center},
    scale = 0.55,
    ylabel= {$\#$ GMRES \\ $\epsilon=10^{-5}$},
    xlabel = {$n$},
    xmin=90, xmax=3200,
    ymin=2, ymax=300,
    xtick = \empty,
    extra x ticks = {127, 255, 511, 1023, 2047},
    extra x tick labels = {127,  ,511, , 2047},
    ymajorgrids=true,
    line width=0.25mm,
    grid style=dashed
]
\addplot coordinates {
    (127, 3)(255, 3) (511,3) (1023, 3) (2047, 4)
    };
\addplot coordinates {
    (127, 8)(255,13) (511, 23) (1023, 46) (2047, 91)
    };
\end{loglogaxis}
\end{tikzpicture}}
\scalebox{0.75}{
\begin{tikzpicture}    
\begin{loglogaxis}[
    scale = 0.55,
    xmin=90, xmax=3200,
    ymin=2, ymax=300,
    xtick = \empty,
    yticklabel = \empty,
    xlabel = $n$,
    extra x ticks = {127, 255, 511, 1023, 2047},
    extra x tick labels = {127,  ,511, , 2047},
    ymajorgrids=true,
    line width=0.25mm,
    grid style=dashed
]
\addplot coordinates {
    (127, 3)(255,3) (511, 4) (1023, 4) (2047,4)
    };
\addplot coordinates {
    (127, 11)(255,27) (511,63) (1023,123 ) (2047, 267)
    };
\end{loglogaxis}
\end{tikzpicture}}
\scalebox{0.75}{
\begin{tikzpicture}
\pgfplotsset{
every axis legend/.append style={
at={(1.02,1)},
anchor=north west,
text width=40pt},
legend cell align=left}
\begin{loglogaxis}[
legend columns = 1,
legend style = {draw = none},
    scale = 0.55,
    xmin=90, xmax=3200,
    ymin=2, ymax=300,
    xtick = \empty,
    extra x ticks = {127, 255, 511, 1023, 2047},
    extra x tick labels = {127,  ,511, , 2047},
    ymajorgrids=true,
    yticklabel = \empty,
    xlabel = $n$,
    line width=0.25mm,
    grid style=dashed,
    legend entries  = {spaQR, spaQR w/o
    scaling}
]
\addplot coordinates {
    (127, 3)(255, 4) (511, 5) (1023, 5) (2047, 5)
    };
\addplot coordinates {
    (127, 13)(255,57) (511, 259) (1023, nan) (2047, nan)
    };
\end{loglogaxis}
\end{tikzpicture}}
\end{subfigure}

\caption{Comparison of the spaQR algorithm with and without scaling on 2D $n \times n$ High Contrast Advection Diffusion problems for three values of the parameter $\rho$. The two variations of the spaQR algorithm are compared for two values of the tolerance $\epsilon=10^{-3}$, $10^{-5}$. Notice that the spaQR algorithm (with scaling) outperforms the variant without scaling in all the cases. Moreover, for small enough $\epsilon$, spaQR algorithm converges in a constant number of iterations irrespective of the problem size for three values of the parameter $\rho$.}
\label{fig:hc_ad_comparison}
\end{figure}

In \Cref{fig:hc_ad_comparison}, we compare the number of GMRES iterations needed to converge by the two variants of the algorithm for three values of the parameter $\rho$. The problem becomes increasingly ill-conditioned as the parameter $\rho$ increases. The spaQR algorithm (with scaling) performs much better as compared to the variant without block diagonal scaling. For small enough tolerance $\epsilon$, the convergence of the spaQR algorithm is independent of the problem size $N=n^2$.

\subsubsection{Non-regular problems}

Next, we test the two variants of the spaQR algorithm on a set of matrices taken from the SuiteSparse Matrix Collection~\cite{suitesparse}. The name of the matrices and their properties such as the size, the number of non-zero entries, pattern symmetry, numerical symmetry and the application domain are given in \Cref{Table: suite sparse}. The matrices are partitioned using the modified HUND and row ordering is performed based on the heuristics discussed in \Cref{ord_clus}. 

The number of GMRES iterations taken by the two variants of the spaQR algorithm for the ten matrices listed in \Cref{Table: suite sparse} are given in \Cref{Table: suite_sparse_results}. In nine out of the ten cases, spaQR algorithm (with scaling) performs better than the variant without block diagonal scaling. With a lower tolerance of $\epsilon = 10^{-6}$, both variants have almost the same performance.

\begin{table}[tbhp]
    \centering
    \label{Table: suite sparse}
    \caption{List of test matrices and their properties: number of rows and columns (size), number of non-zeros (nnz), pattern symmetry (pat. sym.), numerical symmetry (num. sym.) and the problem domain (Kind). }
    \begin{tabular}{rrrrrrp{115pt}}
        \toprule
        \# & Matrix & size & nnz & Pat. & Num. & Kind\\
        & & & & sym. & sym. & \\
        \midrule
        1 & cavity15 & 2195 & 71601 & 5.9 & 0.0 & Subsequent CFD Problem \\
        2 & cavity26 & 4562 & 138187 & 5.9 & 0.0 & Subsequent CFD Problem \\
        3 & dw4096 & 8192 & 41746 & 96.3 & 91.5 & Electromagnetics problem \\
        4 & Goodwin\_030 & 10142 & 312814 & 96.6 & 6.3 & CFD problem \\
        5 & inlet & 11730 & 328323 & 60.8 & 0 & Model Reduction Problem \\
        6 & Goodwin\_040 & 17922 & 561677 & 97.5 & 6.4 & CFD problem \\
        7 & wang4 & 26068 & 177196 & 100 & 4.6 & Semiconductor device problem \\
        8 & Zhao1 & 33381 & 166453 & 92.2 &0.0 & Electromagnetics problem\\
        9 & Chevron1 & 37365 & 330633 & 99.5 & 71.0 & Seismic modelling \\
        10 & cz40948 & 40948 & 412148 & 43.5 & 23.7 & Closest Point Method \\
        \bottomrule
    \end{tabular}
\end{table}

\subsubsection{2D flow in a driven cavity}

The lid-driven flow in a cavity is a well-studied problem. The problem deals with a viscous incompressible fluid flow in a square cavity. The cavity consists of three rigid walls with no-slip conditions and a lid moving with tangential unit velocity. This results in a circular flow. 
\begin{table}[tbhp]
    \centering
        \caption{Performance of the spaQR algorithm with and without scaling in terms of the number of GMRES iterations needed to converge. The test problems are listed in \Cref{Table: suite sparse}.}
    \label{Table: suite_sparse_results}
    \begin{tabular}{@{}rrrrr@{}}
        \toprule
         & \multicolumn{2}{c}{\# GMRES, $\epsilon=10^{-3}$} & \multicolumn{2}{c}{\# GMRES, $\epsilon=10^{-6}$} \\
        
        \cmidrule(lr){2-3} 
        \cmidrule(lr){4-5} 
        \# & spaQR & spaQR & spaQR & spaQR \\
         & & w/o scaling & & w/o scaling\\
        \midrule
        1 & 58 & \textbf{43} & \textbf{5} & 10\\
        2 & \textbf{25} &87& \textbf{4}& 11\\
        3 & \textbf{23} & 45 & 4 & 4\\
        4 & \textbf{7} & 16 & \textbf{3} & 4\\
        5 & \textbf{75} & 138 & \textbf{5} & 7\\
        6 & \textbf{7} & 22 & \textbf{3} & 4 \\
        7 & \textbf{6} & 17 & \textbf{3} & 4 \\
        8 & \textbf{6} & 7 & 5 & 5\\
        9 & \textbf{21} & 108 & \textbf{4} & 6\\
        10 & \textbf{5} & 77 & \textbf{2} & 9 \\
        \bottomrule
    \end{tabular}
\end{table}

The matrices arising from this problem are real and unsymmetric (symmetric indefinite in the case of $\text{Re}=0$). They are good test cases for iterative solvers as they are difficult to solve without an efficient preconditioner~\cite{Boisvert1997}. Incomplete LU based preconditioners fail on these matrices. They are unstable due to singular pivots. The spaND algorithm also fails on these matrices for the same reasons. 

On the other hand, spaQR provides increased stability and the spaQR preconditioned system converges in less than 50 GMRES iterations for a wide range of Reynolds number. The matrices used for testing are taken from the SPARSKIT collection~\cite{Boisvert1997} and have a size of 17,281 with 553,956 non-zero entries. The performance of the two variants of spaQR algorithm in terms of the number of GMRES iterations needed to converge are shown in \Cref{Table: cavity_flow}  for $0 \leq \text{Re} \leq 5000$.  spaQR algorithm (with scaling) outperforms the variant without scaling for the entire range of Reynolds number tested. However, neither of the two variants break down during the factorization phase. 

\begin{table}[tbhp]
    \centering
        \caption{Performance of spaQR algorithm on 2D fluid flow in a driven cavity. spaQR w/o scaling failed to converge in less than 300 iterations for the last two matrices. }
    \label{Table: cavity_flow}
    \begin{tabular}{@{}crcc@{}}
        \toprule
         & & \multicolumn{2}{c}{\# GMRES, $\epsilon=10^{-5}$}  \\
        
        \cmidrule(lr){3-4} 
        Matrix & Re & spaQR & spaQR  \\
        & & & w/o scaling\\
        \midrule
        E40R0000 & 0& 6 & 39 \\
        E40R0100 &100 & 7 & 42 \\
        E40R0500 &500 & 6 & 46\\
        E40R1000 & 1000 & 11 & 62 \\
        E40R2000 & 2000 & 23 & 138\\
        E40R3000 & 3000 & 19 & 225 \\
        E40R4000 & 4000 & 36 & ---\\
        E40R5000 & 5000 & 21 & ---\\
        \bottomrule
    \end{tabular}

\end{table}

Along with the theoretical results on scaling (see \Cref{theoretical_results_sec}), the numerical experiments show that, in general, scaling is advantageous and leads to better performance. However, scaling should be used with caution for highly ill-conditioned problems. For these problems, scaling can only be done on alternate levels or can be done based on the condition number of the diagonal blocks. This is a topic for future research. In the rest of the section, we only consider the variant with block diagonal scaling (spaQR). 

\subsection{Scaling with problem size} Next, we study the variation in the time to build the preconditioner and the number of GMRES iterations with the problem size on 2D and 3D Advection Diffusion problems.

\subsubsection{2D Advection Diffusion problem}

Let us consider the variable coefficient advection diffusion equation with $a(\mathbf{x})=1$. The constant $q$ controls the magnitude of the convective term. The equation is discretized on a uniform $n \times n$ 2D grid using the centered finite difference scheme.  The resulting linear system becomes strongly unsymmetric as the convective term becomes dominant (higher value of $q$) and hence, is challenging to solve. We test the performance of our algorithm on these problems with different parameters $b(\mathbf{x})$, $q$ with  $a(\mathbf{x})$ fixed at $1$ .  The spaQR algorithm is used as a preconditioner to accelerate the convergence of the GMRES iterative solver. 

\begin{figure}[tbhp]
\centering
\begin{subfigure}[t]{\textwidth}
\centering
\scalebox{0.75}{
\begin{tikzpicture}
\begin{loglogaxis}[
    scale = 0.75,
    ylabel={$\#$ GMRES},
    xlabel = {$N$},
    ymin = 2, ymax = 100,
    xtick = \empty,
    extra x ticks = {16000, 6.5*10^4, 2.5*10^5, 10^6, 4*10^6},
    extra x tick labels = {16k, , 0.25M, ,4M},
    ymajorgrids=true,
    line width=0.25mm,
    grid style=dashed,
]
\addplot coordinates {
    (128^2, 9)(256^2, 11) (512^2, 14) (1024^2, 17) (2048^2, 23)
    };
\addplot coordinates {
    (128^2, 8)(256^2, 9) (512^2, 11) (1024^2, 13) (2048^2, 16)
    };
\addplot+[mark = triangle*] coordinates {
    (128^2, 8)(256^2, 9) (512^2, 9) (1024^2, 11) (2048^2, 13)
    };
\end{loglogaxis}
\end{tikzpicture}}
\hspace{0.2cm}
\scalebox{0.75}{
\begin{tikzpicture}
\pgfplotsset{
every axis legend/.append style={
at={(1.02,1)},
anchor=north west},
legend cell align=left}
        \begin{loglogaxis}[
        legend columns = 1,
        legend style = {draw = none},
            scale = 0.75,
            xlabel = {$N$},
            ylabel = {Time to factorize ($s$)},
            ymin = 0.1, ymax = 100,
            xtick = \empty,
            extra x ticks = {16000, 6.5*10^4, 2.5*10^5, 10^6, 4*10^6},
            extra x tick labels = {16k, , 0.25M, ,4M},
            ymajorgrids=true,
            line width=0.25mm,
            grid style=dashed,
            legend entries = {$q = 1$, $q=25$, $q=1000$}
        ]
\addplot coordinates {
    (16384, 0.29)(65536, 0.753) (262144, 2.83) (1048576, 13.5) (4194304, 54.25)
    };
\addplot coordinates {
    (128^2, 0.19)(256^2, 0.786) (512^2, 2.769) (1024^2, 12.8) (2048^2, 47.8)
    };
\addplot+[mark = triangle*] coordinates {
    (128^2, 0.233)(256^2, 0.796) (512^2, 2.884) (1024^2, 12.75) (2048^2, 48)
    };
\addplot [black, domain = 128^2:2048^2] {x/150000};
\node [ anchor=center] at (3.5*10^5,1) {$\mathcal{O}(N)$};
        \end{loglogaxis}
    \end{tikzpicture}}
\end{subfigure}
\caption{Results for the 2D advection diffusion problem for varying values of $q$. The threshold $\epsilon$ for ignoring singular values in the spaQR algorithm is $\epsilon = 10^{-2}$. Note that the number of iterations grows slowly and the factorization time scales linearly with problem size for all three values of $q$.}
\label{ad_figure}
\end{figure}
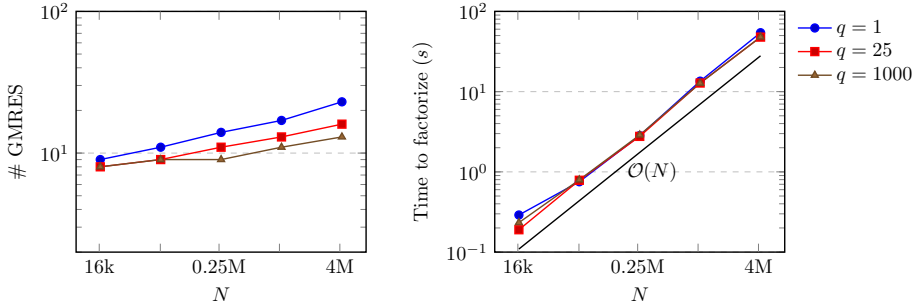
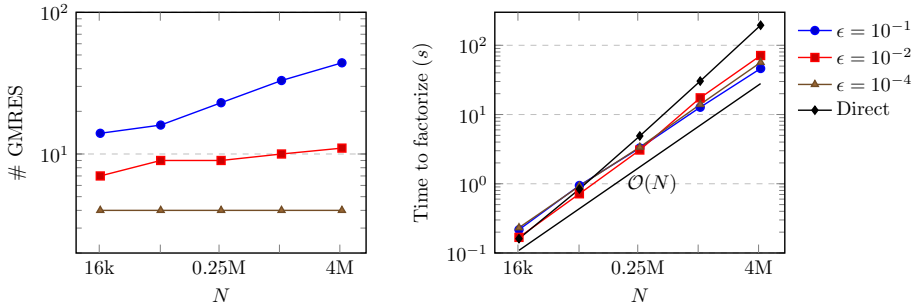
\begin{figure}[tbhp]
\centering
\begin{subfigure}[t]{\textwidth}
\centering
\scalebox{0.75}{
\begin{tikzpicture}
\begin{loglogaxis}[
    scale = 0.75,
    ylabel={$\#$ GMRES},
    xlabel = {$N$},
    ymin = 2, ymax = 100,
    xtick = \empty,
    extra x ticks = {16000, 6.5*10^4, 2.5*10^5, 10^6, 4*10^6},
    extra x tick labels = {16k, , 0.25M, ,4M},
    ymajorgrids=true,
    line width=0.25mm,
    grid style=dashed
]
\addplot coordinates {
    (128^2, 14)(256^2, 16) (512^2, 23) (1024^2, 33) (2048^2, 44)
    };
\addplot coordinates {
    (128^2, 7)(256^2, 9) (512^2, 9) (1024^2, 10) (2048^2, 11)
    };
\addplot+[mark = triangle*] coordinates {
    (128^2, 4)(256^2, 4) (512^2, 4) (1024^2, 4) (2048^2, 4)
    };
\end{loglogaxis}
\end{tikzpicture}}
\hspace{0.2cm}
\scalebox{0.75}{
\begin{tikzpicture}
\pgfplotsset{
every axis legend/.append style={
at={(1.02,1)},
anchor=north west},
legend cell align=left}
        \begin{loglogaxis}[
        legend columns = 1,
        legend style = {draw = none},
            scale = 0.75,
            xlabel = {$N$},
            ylabel = {Time to factorize ($s$)},
            ymin = 0.1, ymax = 300,
            xtick = \empty,
            extra x ticks = {16000, 6.5*10^4, 2.5*10^5, 10^6, 4*10^6},
            extra x tick labels = {16k, , 0.25M, ,4M},
            ymajorgrids=true,
            line width=0.25mm,
            grid style=dashed,
              legend entries = {$\epsilon = 10^{-1}$, $\epsilon = 10^{-2}$, $\epsilon = 10^{-4}$,
              Direct}]
\addplot coordinates {
    (128^2, 0.217)(256^2, 0.944) (512^2, 3.29) (1024^2, 12.73) (2048^2, 46.2)
    };
\addplot coordinates {
    (128^2, 0.168)(256^2, 0.716) (512^2, 3.07) (1024^2, 17.42) (2048^2, 71.2)
    };
\addplot+[mark = triangle*] coordinates {
    (128^2, 0.232)(256^2, 0.924) (512^2, 3.38) (1024^2, 14) (2048^2, 56.6)
    };
\addplot+[mark = diamond*] coordinates {
    (128^2, 0.161)(256^2, 0.838) (512^2, 4.9) (1024^2, 30.4) (2048^2, 194.8)
    };
\addplot [black, domain = 128^2:2048^2] {x/150000};
\node [ anchor=center] at (3.5*10^5,1) {$\mathcal{O}(N)$};
        \end{loglogaxis}
    \end{tikzpicture}}
\end{subfigure}
\caption{Variation in the number of iterations and time to factorize with tolerance $\epsilon$ for the 2D  advection diffusion problem with $a = 1$, $b(x,y) = e^{x+y}$, $q=1000$. The iteration count is constant for small enough tolerance $\epsilon$ and the factorization time scales linearly with the problem size. The direct method with the same partition scales as $\mathcal{O}(N^{3/2})$.}
\label{ad_eps}
\end{figure}

\Cref{ad_figure} compares the number of GMRES iterations needed for convergence and the time taken to factorize for the 2D advection diffusion problem with $a=1$, $b=1$, and $q=1$, 25, 1000. The time to factorize the matrix scales as $\mathcal{O}(N)$ in contrast to Nested Dissection Householder QR which scales as $\mathcal{O}(N^{3/2})$. Combining this with the slow increase in the number of iterations to converge, gives a approximate complexity of $\mathcal{O}(N)$ complexity to the algorithm.

In \Cref{ad_eps}, we compare the iteration count and time to factorize for various values of the tolerance $\epsilon$. Note that the time to factorize scales as $\mathcal{O}(N)$ independent of the value of $\epsilon$ used. The rate of convergence of the residual $\|Ax-b\|_{_2}/\|b\|_{_2}$ with the GMRES iterations is shown in \Cref{fig:2d_ad_gmres_residual}. The rate of convergence of the residual increases greatly as the tolerance $\epsilon$ is decreased from $10^{-1}$ to $10^{-4}$. The optimal value of $\epsilon$ depends on the problem and is to be chosen such that the overall time (factorization $+$ solve) is minimized. 

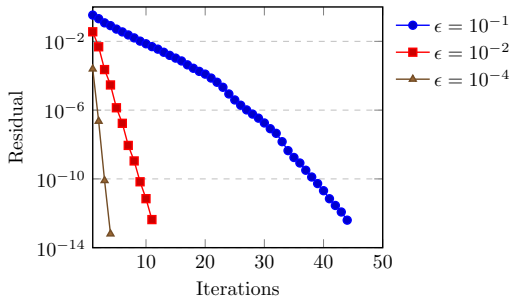
\begin{figure}[tbhp]
    \centering
    \scalebox{0.75}{
    \begin{tikzpicture}
\pgfplotsset{
every axis legend/.append style={
at={(1.02,1)},
anchor=north west},
legend cell align=left}
\begin{semilogyaxis}[
legend columns = 1,
legend style = {draw = none},
    scale = 0.75,
    ylabel={Residual},
    xlabel = {Iterations},
    ymin = 1e-14, ymax =1,
    xmin=1, xmax=50,
    ymajorgrids=true,
    line width=0.25mm,
    grid style=dashed,
    legend entries = {$\epsilon = 10^{-1}$, $\epsilon = 10^{-2}$, $\epsilon = 10^{-4}$}
]
\addplot table {results/2d_ad_q1000_p3_t0_1.dat};
\addplot table {results/2d_ad_q1000_p3_t0_01.dat};
\addplot+[mark = triangle*] table {results/2d_ad_q1000_p3_t0_0001.dat};
\end{semilogyaxis}
\end{tikzpicture}}
    \caption{The convergence of the residual $\|Ax-b\|_{_2}/\|b\|_{_2}$ with the number of GMRES iterations for different values of the tolerance $\epsilon$ for the 2D advection diffusion problem on the $2048 \times 2048$ grid.}
    \label{fig:2d_ad_gmres_residual}
\end{figure}

\subsubsection{3D Advection Diffusion problem}
\begin{figure}[tbhp]
\centering
\begin{subfigure}[t]{\textwidth}
\centering
\scalebox{0.75}{
\begin{tikzpicture}

\begin{loglogaxis}[
    scale = 0.75,
    ylabel={$\#$ GMRES},
    xlabel = {$N$},
    ymin = 5, ymax = 500,
    xmin=200000, xmax=18000000,
    xtick = {2.5*10^5, 5*10^5, 10^6, 2*10^6, 4*10^6, 8*10^6, 16*10^6},
    xticklabels = {0.25M, ,1M, ,4M, ,16M},
    ymajorgrids=true,
    line width=0.25mm,
    grid style=dashed
]
\addplot coordinates {
    (262144, 68) (512000, 87)
    (884736, 102) (2097152, 130 ) (4096000, 162) (7077888, 199) (16777216, nan)
    };

\addplot coordinates {
    (262144, 12) (512000, 14) (884736, 14) (2097152, 14) (4096000, 17) (7077888, 19) (16777216, nan)
    };
\end{loglogaxis}
\end{tikzpicture}}
\scalebox{0.75}{
    \begin{tikzpicture}
\pgfplotsset{
every axis legend/.append style={
at={(1.02,1)},
anchor=north west},
legend cell align=left}
        \begin{loglogaxis}[
        legend columns = 1,
        legend style = {draw = none},
            scale = 0.75,
            xlabel = {$N$},
            ylabel = {Time to factorize ($s$)},
            ymin = 30, ymax = 400000,
            xmin=200000, xmax=18000000,
            xtick = {2.5*10^5, 5*10^5, 10^6, 2*10^6, 4*10^6, 8*10^6, 16*10^6},
            xticklabels = {0.25M, ,1M, ,4M, ,16M},
            ymajorgrids=true,
            line width=0.25mm,
            grid style=dashed,
              legend entries = {$\epsilon = 10^{-1}$,  $\epsilon = 10^{-2}$, Direct}
        ]
\addplot coordinates {
    (262144, 42.07) (512000, 101.8 )
    (884736, 211) (2097152,651 ) (4096000, 1616.8) (7077888, 3462.4) (16777216, 14206)
    };

\addplot coordinates {
    (262144, 140) (512000, 339.58)
    (884736, 701.8) (2097152, 2298) (4096000, 6217) (7077888,15503) (16777216, nan)
    };
\addplot+[draw=black, mark = diamond*, mark options = {fill=black}] coordinates {
    (262144, 411.3) (512000, 1618) (884736, 5471) (2097152, nan) (4096000, nan) (7077888, nan) (16777216, nan)
    };
\addplot[draw=black, dashed, mark= diamond, mark options={solid}] coordinates {
    (262144, nan) (512000, nan) (884736, 5471) (2097152, 32000) (4096000, 122070) (7077888, 364500) (16777216, nan)
};

\addplot [black, domain = 250000:16000000] {x*log2(x)/150000};
\node [anchor=center] at (5000000,170) {$\mathcal{O}(N\log N)$};
        \end{loglogaxis}
    \end{tikzpicture}}
\end{subfigure}
\caption{Variation in the number of iterations and time to factorize with tolerance $\epsilon$ for the 3D $n \times n \times n$ advection diffusion problem with $a=1$, $b=1$, $q=1$. The iteration count increases slowly for small enough tolerance $\epsilon$. Empirically,  the factorization time scales as $\mathcal{O}(N^{1.4})$. The missing data points with spaQR either indicate that the factorization time was more than 5 hours or that GMRES took more than 200 iterations to converge. The scaling of the direct method has been extrapolated for $N= 128^3$, $160^3$, $192^3$. }
\label{Figure:3d_ad}
\end{figure}
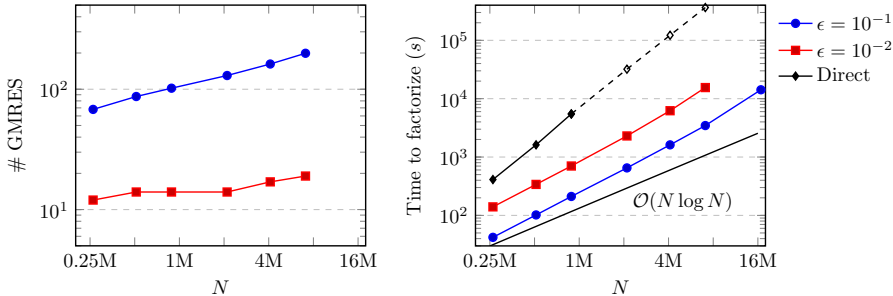

Consider the advection diffusion problem on a uniform $n \times n \times n$ 3D grid. The size of the matrix is $N = n^3$. The performance of the algorithm is reported in terms of the time to factorize and the number of GMRES iterations needed to converge in \Cref{Figure:3d_ad} for various values of tolerance $\epsilon$. Theoretically, we expect the factorization time to scale as $\mathcal{O}(N\log N)$ (see \Cref{sec:complexity}). However, the empirical complexity is $\mathcal{O}(N^{1.4})$. This is likely due to non-asymptotic effects. The convergence of the residual $\|Ax-b\|_{_2}/\|b\|_{_2}$ with the iteration count is shown in \Cref{fig:3d_ad_gmres_residual} for $N=192^3$. Similar to the 2D case, we notice that rate of convergence of the residual increases drastically as the tolerance $\epsilon$ is decreased from $10^{-1}$ to $10^{-2}$.

\begin{figure}[tbhp]
    \centering
    \scalebox{0.75}{
    \begin{tikzpicture}
\pgfplotsset{
every axis legend/.append style={
at={(1.02,1)},
anchor=north west},
legend cell align=left}
\begin{semilogyaxis}[
legend columns = 1,
legend style = {draw = none},
    scale = 0.75,
    ylabel={Residual},
    xlabel = {Iterations},
    ymin = 1e-14, ymax =1,
    xmin=1, xmax=200,
    ymajorgrids=true,
    line width=0.25mm,
    grid style=dashed,
    legend entries = { $\epsilon = 10^{-1}$, $\epsilon = 10^{-2}$}
]

\addplot table {results/3d_ad_192_t0_1.dat};
\addplot+[mark = triangle*] table {results/3d_ad_192_t0_01.dat};
\end{semilogyaxis}
\end{tikzpicture}}
    \caption{Convergence of the residual $\|Ax-b\|_{_2}/\|b\|_{_2}$ with the number of GMRES iterations for different values of tolerance $\epsilon$ for the 3D advection diffusion problem on the $192 \times 192 \times 192$ grid.}
    \label{fig:3d_ad_gmres_residual}
\end{figure}
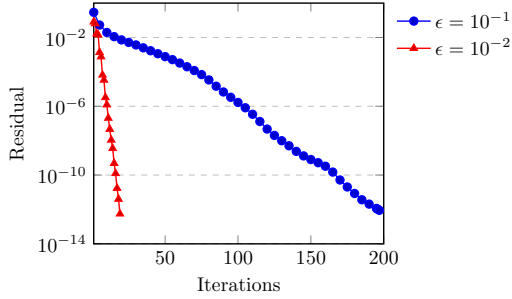
\begin{figure}[tbhp]
\centering
\scalebox{0.75}{
\begin{tikzpicture}
\pgfplotsset{
every axis legend/.append style={
at={(1.02,1)},
anchor=north west},
legend cell align=left}
\begin{semilogyaxis}[
legend columns = 1,
legend style = {draw = none},
    scale = 0.75,
    ylabel={$|R_{ii}|/|R_{11}|$},
    ymin = 1e-8, ymax = 1,
    xmin=1, xmax=350,
    ymajorgrids=true,
    line width=0.25mm,
    grid style=dashed,
    legend entries = { $l=8$, $l=6$, $l=4$, $l=2$}
]
\addplot+ table [x expr=\coordindex, y index=0]  {results/3d_64_sv_3.dat};
\addplot+ table [x expr=\coordindex, y index=0]  {results/3d_64_sv_5.dat};
\addplot+ table [x expr=\coordindex, y index=0]  {results/3d_64_sv_7.dat};
\end{semilogyaxis}
\end{tikzpicture}}
\caption{The singular value decay of the block $\begin{bmatrix} A_{np}^T & A_{pn} \end{bmatrix}$ corresponding to an interface $p$ of the top separator at various levels of sparsification. The diagonal entries $|R_{ii}|$ of a column pivoted QR on the block is used as a substitute for the singular values. The results shown are on the 3D advection diffusion problem with $N=64^3$.}
\label{fig:sv_decay}
\end{figure}
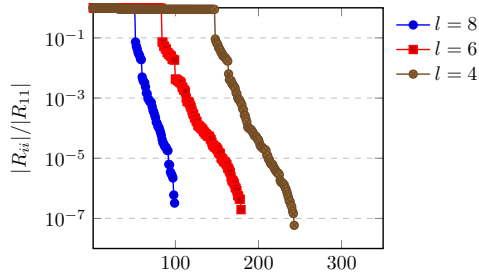
\subsection{Profiling}
\label{Sec: Profiling}

\begin{figure}[tbhp]
\centering
\begin{subfigure}{\textwidth}
\centering
\scalebox{0.75}{
\begin{tikzpicture}
        \begin{semilogyaxis}[
            scale = 0.7,
            ylabel={Size of interface},
            ymin = 25, ymax = 5000,
            xmin=1, xmax = 12,
            xticklabel = \empty,
            x dir=reverse,
            enlarge x limits=0.1,
            bar width = 8pt,
            ymajorgrids=true,
            grid style=dashed,
            line width=0.25mm,
            xtick align = inside
        ]
        \addplot[ybar, draw=blue, fill=blue!30, error bars/.cd,
        y explicit,
        y dir=both,
        error bar style={line width=0.3mm, black}] 
                    table [
                    x = level,
                    y = median,
                    y error plus expr=\thisrow{q2}-\thisrow{median},
                    y error minus expr=\thisrow{median}-\thisrow{q1}
                    ] 
                    {results/3d_64_ranks.dat};
        \addplot [red, dashed, domain = 1:10] {2^(-(x-1)/3)*700};
        \end{semilogyaxis}
\end{tikzpicture}}
\scalebox{0.75}{
    \begin{tikzpicture}
        \begin{semilogyaxis}[
            scale = 0.7,
            ymin = 25, ymax = 5000,
            xmin=1, xmax = 15,
            xticklabel = \empty,
            x dir=reverse,
            yticklabel = \empty,
            enlarge x limits=0.07,
            bar width = 7pt,
            ymajorgrids=true,
            line width =0.25mm,
            grid style=dashed,
            xtick align = inside
        ]
        \addplot[ybar, draw=blue, fill=blue!30,error bars/.cd,
                    y explicit,
                    y dir=both, 
                    error bar style={line width=0.3mm, black}] 
                    table [
                    x = level,
                    y = median,
                    y error plus expr=\thisrow{q2}-\thisrow{median},
                    y error minus expr=\thisrow{median}-\thisrow{q1}
                    ] 
                    {results/3d_128_ranks.dat};
        \addplot [red, dashed, domain = 1:13] {2^(-(x-1)/3)*1200};
        \end{semilogyaxis}
    \end{tikzpicture}}
\scalebox{0.75}{
    \begin{tikzpicture}
        \begin{semilogyaxis}[
            scale = 0.7,
            ymin = 25, ymax = 5000,
            xmin=1, xmax = 18,
            xticklabel = \empty,
            x dir=reverse,
            yticklabel = \empty,
            enlarge x limits=0.06,
            ymajorgrids=true,
            line width=0.25mm,
            grid style=dashed,
            xtick align = inside,
            bar width = 5pt
        ]
        \addplot[ybar, draw=blue, fill=blue!30,  error bars/.cd,
                    y explicit,
                    y dir=both, 
                    error bar style={line width=0.3mm, black} ]
                    table [
                    x = level,
                    y = median,
                    y error plus expr=\thisrow{q2}-\thisrow{median},
                    y error minus expr=\thisrow{median}-\thisrow{q1},
                    ]  {results/3d_256_ranks.dat};
    \addplot [red, dashed, domain = 1:16] {2^(-(x-1)/3)*2400};
        \end{semilogyaxis}
    \end{tikzpicture}}
\end{subfigure}

\begin{subfigure}{\textwidth}
\centering
\scalebox{0.75}{
\begin{tikzpicture}
        \begin{semilogyaxis}[
        ylabel style = {align=center},
            scale = 0.7,
            ylabel= {Median non-zeros},
            ymin = 900, ymax = 20000,
            xmin=1, xmax = 12,
            xticklabel = \empty,
            x dir=reverse,
            enlarge x limits=0.1,
            bar width = 8pt,
            ymajorgrids=true,
            grid style=dashed,
            line width=0.25mm,
            xtick align = inside
        ]
        \addplot[ybar, draw=blue, fill=blue!30, error bars/.cd,
        y explicit,
        y dir=both,
        error bar style={line width=0.3mm, black}] 
                    table [
                    x = level,
                    y = median,
                    y error plus expr=\thisrow{q2}-\thisrow{median},
                    y error minus expr=\thisrow{median}-\thisrow{q1}
                    ] 
                    {results/3d_64_nbrs.dat};
    \addplot [red, dashed, domain = 4:10] {2^(-(x-1)/3)*612*24};
        \end{semilogyaxis}
\end{tikzpicture}}
\scalebox{0.75}{
    \begin{tikzpicture}
        \begin{semilogyaxis}[
            scale = 0.7,
            ymin = 900, ymax = 20000,
            xmin=1, xmax = 15,
            xticklabel = \empty,
            x dir=reverse,
            yticklabel = \empty,
            enlarge x limits=0.07,
            bar width = 7pt,
            ymajorgrids=true,
            line width =0.25mm,
            grid style=dashed,
            xtick align = inside
        ]
        \addplot[ybar, draw=blue, fill=blue!30,error bars/.cd,
                    y explicit,
                    y dir=both, 
                    error bar style={line width=0.3mm, black}] 
                    table [
                    x = level,
                    y = median,
                    y error plus expr=\thisrow{q2}-\thisrow{median},
                    y error minus expr=\thisrow{median}-\thisrow{q1}
                    ] 
                    {results/3d_128_nbrs.dat};
        \addplot [red, dashed, domain = 5:13] {2^(-(x-1)/3)*1124*30};
        \end{semilogyaxis}
    \end{tikzpicture}}
\scalebox{0.75}{
    \begin{tikzpicture}
        \begin{semilogyaxis}[
            scale = 0.7,
            ymin = 900, ymax = 20000,
            xmin=1, xmax = 18,
            xticklabel = \empty,
            x dir=reverse,
            yticklabel = \empty,
            enlarge x limits=0.06,
            bar width = 5pt,
            ymajorgrids=true,
            line width=0.25mm,
            grid style=dashed,
            xtick align = inside
        ]
        \addplot[ybar, draw=blue, fill=blue!30,  error bars/.cd,
                    y explicit,
                    y dir=both, 
                    error bar style={line width=0.3mm, black}]
                    table [
                    x = level,
                    y = median,
                    y error plus expr=\thisrow{q2}-\thisrow{median},
                    y error minus expr=\thisrow{median}-\thisrow{q1}
                    ] 
                    {results/3d_256_nbrs.dat};
        \addplot [red, dashed, domain = 1:16] {2^(-(x-1)/3)*2348*30};
        \end{semilogyaxis}
    \end{tikzpicture}}
\end{subfigure}

\begin{subfigure}{\textwidth}
\centering
\scalebox{0.75}{
\begin{tikzpicture}
        \begin{semilogyaxis}[
        xlabel style = {align=center},
            scale = 0.7,
            ybar,
            ylabel={Time to sparsify (s)},
            xlabel = {Level \\
            $N = 64^3$},
            ymin = 1, ymax = 2000,
            xmin=1, xmax = 12,
            xticklabel = \empty,
            x dir=reverse,
            extra x ticks = { 12, 11, 10, 9,8,7,6,5,4,3,2,1 },
            extra x tick labels = {12,,10,,8,,6,,4,,2,},
            enlarge x limits=0.1,
            bar width = 8pt,
            ymajorgrids=true,
            grid style=dashed,
            line width=0.25mm,
            xtick align = inside
        ]
        \addplot+    table          {results/3d_64_sparsify_time.dat};
        \end{semilogyaxis}
\end{tikzpicture}}
\scalebox{0.75}{
    \begin{tikzpicture}
        \begin{semilogyaxis}[
        xlabel style = {align=center},
            scale = 0.7,
            ybar,
            ymin = 1, ymax = 2000,
            xmin=1, xmax = 15,
            xlabel={Level \\
            $N=128^3$},
            xticklabel = \empty,
            x dir=reverse,
            extra x ticks = { 15,14,13,12, 11, 10, 9,8,7,6,5,4,3,2,1 },
            extra x tick labels = {15,,13,,11,,9,,7,,5,,3,,1},
            yticklabel = \empty,
            enlarge x limits=0.07,
            bar width = 7pt,
            ymajorgrids=true,
            line width =0.25mm,
            grid style=dashed,
            xtick align = inside
        ]
        \addplot+ table  {results/3d_128_sparsify_time.dat};
        \end{semilogyaxis}
    \end{tikzpicture}}
\scalebox{0.75}{
    \begin{tikzpicture}
        \begin{semilogyaxis}[
        xlabel style = {align=center},
            scale = 0.7,
            ybar,
            ymin = 1, ymax = 2000,
            xmin=1, xmax = 18,
            xlabel={Level \\
            $N=256^3$},
            xticklabel = \empty,
            x dir=reverse,
            extra x ticks = { 18,17,16,15,14,13,12, 11, 10, 9,8,7,6,5,4,3,2,1 },
            extra x tick labels = {18,,,15,,,12,,,9,,,6,,4,,2,},
            yticklabel = \empty,
            enlarge x limits=0.06,
            bar width = 5pt,
            ymajorgrids=true,
            line width=0.25mm,
            grid style=dashed,
            xtick align = inside
        ]
        \addplot+
                    table  {results/3d_256_sparsify_time.dat};
        \end{semilogyaxis}
    \end{tikzpicture}}
\end{subfigure}
\caption{The median size of an interface, the median number of non-zero entries per row and column (precisely, $\#$ of non-zero columns in $[A_{np}^T \; A_{pn}]$)), and the total time to sparsify the interfaces per level is shown for the 3D advection diffusion problem on the $64 \times 64 \times 64$, $128 \times 128 \times 128$ and $256 \times 256 \times 256$ grids. The red dashed line indicates that the interface size and the neighbors vary as $2^{-(l-1)/3}$ as assumed in the complexity analysis. The total time to sparsify has a long plateau at a given problem size. }
\label{fig: 3d_sparsification}
\end{figure}
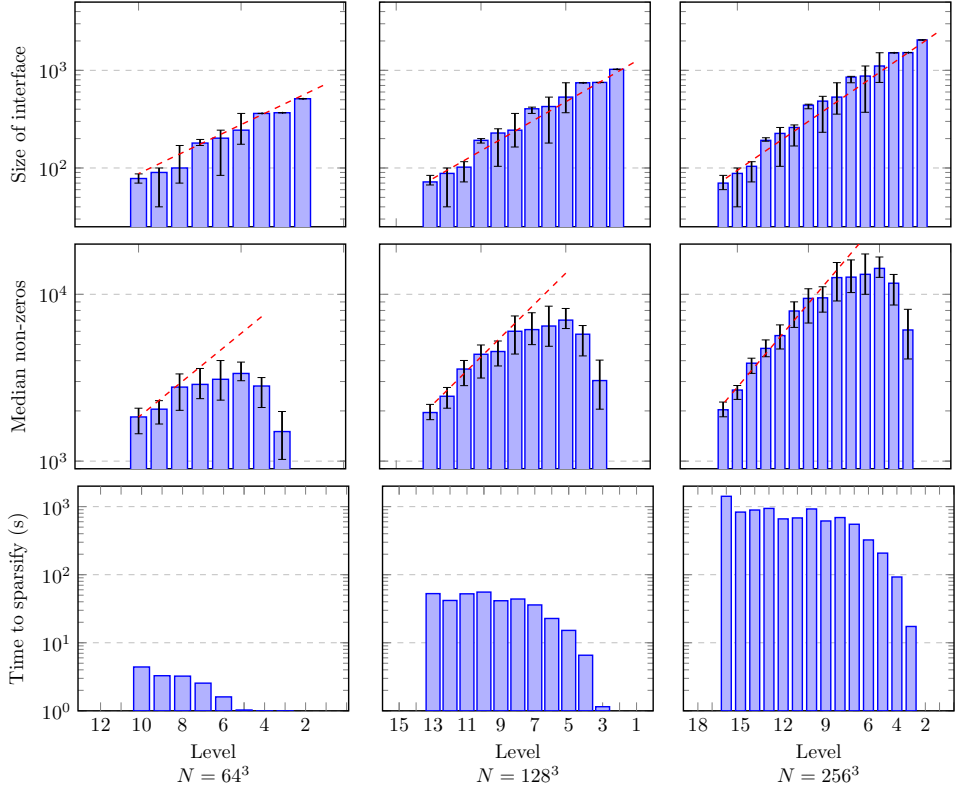

In this section, we give more details on sparsification and the time and memory requirements of the spaQR algorithm. We start with analyzing the singular value decay of a representative block that we compress in \Cref{spars_s} for the 3D advection diffusion problem on the $64 \times 64 \times 64$ grid. \Cref{fig:sv_decay} shows the singular value decay of the block $\begin{bmatrix}
A_{np}^T & A_{pn}
\end{bmatrix}$ corresponding to a representative interface of the top separator at various levels of sparsification. The interface is chosen such that its size is close to the median interface size at that level of sparsification. Roughly, $50\%$ of the singular values are below $\epsilon = 0.1$. Also, note the exponential decay of the singular values after an intial plateau. This observation forms the basis of this work.

Next, we show experimental evidence to back the assumptions made in the complexity analysis. \Cref{fig: 3d_sparsification} shows the median size of an interface ($\#$ rows in $\begin{bmatrix}
A_{np}^T & A_{pn}
\end{bmatrix}$), the number of non-zero rows and columns in the off-diagonal blocks of an interface ($\#$ columns in $\begin{bmatrix}
A_{np}^T & A_{pn}
\end{bmatrix}$), and the total time for sparsification at a given level. The error bars show the inter-quartile range. The red dashed line indicates that the size of the interface grows as $2^{-(l-1)/3}$ where $l$ is the level of the separator of which the interface is a part of. The number of non-zero rows and columns  corresponding to an interface is at most $\mathcal{O}(2^{-(l-1)/3})$ again as indicated by the red dashed line. 

\begin{figure}
    \centering
\scalebox{0.75}{
\begin{tikzpicture}
\begin{loglogaxis}[
    scale = 0.75,
    ylabel={$\text{size}_{\text{top}}$},
    xlabel = {$N$},
    ymin = 400, ymax = 4000,
    xmin=200000, xmax=18000000,
    xtick = {2.5*10^5, 5*10^5, 10^6, 2*10^6, 4*10^6, 8*10^6, 16*10^6},
    xticklabels = {0.25M, ,1M, ,4M, ,16M},
    ymajorgrids=true,
    line width=0.25mm,
    grid style=dashed
]
\addplot coordinates {
    (262144, 512) (512000, 640)
    (884736, 768) (2097152, 1024) (4096000, 1280) (7077888, 1536) (16777216, 2048)
    };

\addplot+[mark = triangle*] coordinates {
    (262144, 823) (512000, 1027) (884736, 1230) (2097152, 1643) (4096000, 2067) (7077888, 2477) (16777216, nan)
    };
\addplot [black, domain = 250000:17000000] {x^(1/3)*7};
\node [ anchor=center] at (5*10^6,900) {$\mathcal{O}(N^{1/3})$};
\end{loglogaxis}
\end{tikzpicture}}
\scalebox{0.75}{
    \begin{tikzpicture}
\pgfplotsset{
every axis legend/.append style={
at={(1.02,1)},
anchor=north west},
legend cell align=left}
        \begin{loglogaxis}[
        legend columns = 1,
        legend style = {draw = none},
            scale = 0.75,
            xlabel = {$N$},
            ylabel = {$\text{mem}_\text{F}$},
            ymin =8*10^7, ymax = 3*10^10,
            xmin=200000, xmax=18000000,
            xtick = {2.5*10^5, 5*10^5, 10^6, 2*10^6, 4*10^6, 8*10^6, 16*10^6},
            xticklabels = {0.25M, ,1M, ,4M, ,16M},
            ymajorgrids=true,
            grid style=dashed,
            line width=0.25mm,
              legend entries = {$\epsilon = 10^{-1}$,  $\epsilon = 10^{-2}$}
        ]
\addplot coordinates {
    (262144, 2*10^8) (512000, 4*10^8 )
    (884736, 7.3*10^8) (2097152, 1.98*10^9) (4096000, 3.93*10^9) (7077888, 7*10^9) (16777216, 1.8*10^10)
    };
\addplot+[mark = triangle*] coordinates {
    (262144, 3.6*10^8) (512000, 7.8*10^8) (884736, 1.48*10^9) (2097152, 4*10^9) (4096000, 8.3*10^9) (7077888, 1.5*10^10) (16777216, nan)
    };

\addplot [black, domain = 250000:17000000] {x*10^3/2};
\node [ anchor=center] at (4*10^6,10^9) {$\mathcal{O}(N)$};
        \end{loglogaxis}
    \end{tikzpicture}}
\caption{The growth in the size of the top separator and the memory required to store the preconditioner with the problem size $N$ for the 3D advection diffusion problem.}
    \label{fig:3d_stop_mem}
\end{figure}
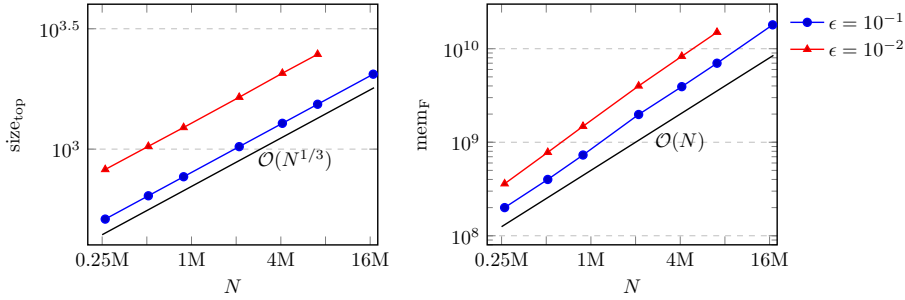

\begin{figure}[tbhp]
    \centering
    \scalebox{0.75}{
    \begin{tikzpicture}
    \pgfplotsset{
every axis legend/.append style={
at={(1.02,1)},
anchor=north west},
legend cell align=left}
        \begin{axis}[
        xlabel style = {align=center},
            xscale = 1.5,
            yscale = 0.75,
            ybar stacked,
            ymin = 1, ymax = 2500,
            xmin=1, xmax = 18,
            xlabel={Level},
            ylabel={Time (s)},
            xticklabel = \empty,
            x dir=reverse,
            extra x ticks = { 18,17,16,15,14,13,12, 11, 10, 9,8,7,6,5,4,3,2,1 },
            extra x tick labels = {18,,,15,,,12,,,9,,,6,,4,,2,},
            enlarge x limits=0.06,
            bar width = 8.5pt,
            ymajorgrids=true,
            line width=0.25mm,
            grid style=dashed,
            xtick align = inside,
            legend entries = {Factorize, Scale, Sparsify, Merge}
        ]
        \addplot[ybar, black, pattern=north east lines] table[
                    x = Level,
                    y = Elimination ] {results/runtime_per_level_256_3d.dat};
        \addplot+[ybar, black, pattern=horizontal lines] table[
                    x = Level,
                    y = Scale] {results/runtime_per_level_256_3d.dat};
        \addplot+[ybar, black, pattern=dots] table[
                    x = Level,
                    y = Sparsify] {results/runtime_per_level_256_3d.dat};
        \addplot+[ybar, black, pattern=vertical lines] table[
                    x = Level,
                    y = Merge] {results/runtime_per_level_256_3d.dat};
        \end{axis}
        \end{tikzpicture}}
    \caption{The runtime per level of the spaQR algorithm split into the four phases: factorize interiors/separators, scale interfaces, sparsify interfaces, and merge the clusters. We skip sparsification for two levels. The results are shown for the 3D advection diffusion problem on a $256 \times 256 \times 256$ grid. }
    \label{fig:runtime_per_level}
\end{figure}
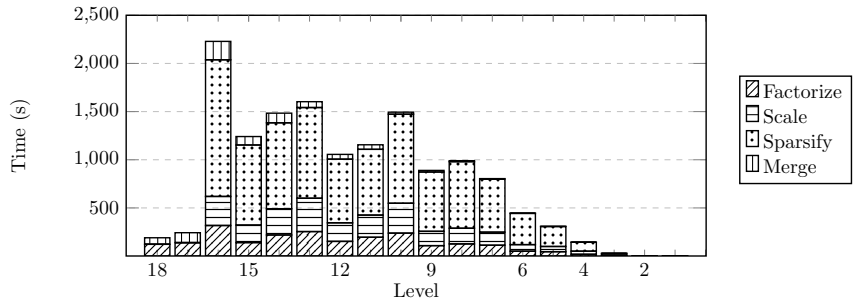

The size of the top separator grows as $\mathcal{O}(N^{1/3})$ as shown in \Cref{fig:3d_stop_mem}. Hence, the cost of factorizing the corresponding block matrix is $\mathcal{O}(N)$. As the cost per level is roughly the same (see \Cref{sec:complexity}) and there are $\Theta(\log(N/N_0))$ levels, this brings the total cost to  $\mathcal{O}(N \log N)$. From \Cref{fig:runtime_per_level}, we see that there is a spike in the runtime at the first level of interface sparsification. Starting sparsification sooner is inefficient as the off-diagonal blocks might not be sufficiently low rank to be beneficial. The runtime in the next few levels have smaller variations which will not matter as we run on bigger matrices. Finally, from \Cref{fig:3d_stop_mem}, we see that the memory required scales as $\mathcal{O}(N)$ as expected.   
\section{Conclusions}

In this work, we develop a novel fast hierarchical QR solver with tunable accuracy for sparse square matrices. We propose an improvement to the base algorithm with a simple block diagonal scaling. We provide theoretical bounds on the error and condition number of the preconditioned matrix. Under certain assumptions (primarily on the required ranks), we proved that the spaQR algorithm scales as $\mathcal{O}(N \log N)$ with a $\mathcal{O}(N)$ solve cost and $\mathcal{O}(N)$ memory requirement. Finally, we provide numerical benchmarks on big sparse unsymmetric linear systems and non-regular problems, which shows the superiority of the algorithm in terms of time and iterations needed to converge to a high accuracy. The additional profiling results give more insight into the algorithm and confirm the validity of the assumptions made in the complexity analysis.

We believe that the spaQR solver opens up exciting new areas that can benefit from fast hierarchical solvers. The algorithm can be extended, with some changes, to rectangular matrices, especially for solving linear least squares problems. This will be investigated in a future work. Further improvements to the algorithm and the implementation are also possible. While the current implementation is sequential, the spaQR algorithm can also be parallelized.

\appendix
\section{Relation to Cholesky}
\label{Related_chol}
An intuitive way to understand sparsification is to consider the relationship between QR and Cholesky. Remember that when $A$ has full column rank, then $R \in \mathbb{R}^{n \times n}$ where $A=QR$ is related to the Cholesky factor $L$ of $A^TA$ by $L = R^T$. 
\[A = QR, \quad  A^TA = R^TQ^TQR = R^TR = LL^T\]
We exploited this relation to use the Nested Dissection ordering of $A^TA$ for performing sparse QR factorization on $A$ and minimize fill-in in $R$. Hence it is necessary that the ND ordering on $A^TA$ is not broken during sparsification. We proved that this is the case for our algorithm in \Cref{well_sep}. In this subsection, we give more intuition behind the algorithm. We discuss different orthogonal transformations that have the potential to sparsify an interface and arrive at the technique used in our spaQR algorithm. 

Consider an interface $p$, its neighbors $n$ (in $G_{A^TA}$) and their associated matrix blocks in $A$ and $A^TA$. Let $w$ be all the remaining nodes. For convenience, denote $S = A^TA$. 
\[A = \begin{bmatrix}
A_{pp} & A_{pn} &  \\
A_{np} & A_{nn} & A_{nw} \\
& A_{wn} & A_{ww}
\end{bmatrix} \quad S = \begin{bmatrix}
S_{pp} & S_{pn} &  \\
S_{np} & S_{nn} & S_{nw} \\
& S_{wn} & S_{ww}
\end{bmatrix}\]
Remember that $S$ is symmetric and $S_{np} = A_{pn}^TA_{pp}+A_{nn}^TA_{np}$. Assume that the off-diagonal blocks in $A$ and $S$ are low rank. 

\subsection*{Sparsification 1}

Consider a low rank approximation of $A_{np}^T$,
\[A_{np}^T = Q_{pp}W_{pn} = \begin{bmatrix}
Q_{pf} & Q_{pc}
\end{bmatrix} \begin{bmatrix}
W_{fn} \\
W_{cn}
\end{bmatrix} \text{ with } \|W_{fn}\|=\mathcal{O}(\epsilon)\] Defining $Q_p = \begin{bmatrix}
Q_{pp} & & \\
& I & \\
& & I
\end{bmatrix}$,
\[AQ_p = \begin{bmatrix}
\Tilde{A}_{ff} & \Tilde{A}_{fc} & A_{fn}  & \\
 \Tilde{A}_{cf} & \Tilde{A}_{cc} & A_{cn} &\\
 \mathcal{O}(\epsilon) & W_{cn}^{T} & A_{nn} & A_{nw} \\ 
 & & A_{wn} & A_{ww}
\end{bmatrix}\quad Q_p^TSQ_p = \begin{bmatrix}
\Tilde{S}_{ff} & \Tilde{S}_{fc} & \Tilde{S}_{fn}  & \\
 \Tilde{S}_{cf} & \Tilde{S}_{cc} & \Tilde{S}_{cn} &\\
 \Tilde{S}_{nf} & \Tilde{S}_{nc} & S_{nn} & S_{nw} \\ 
 & & S_{wn} & S_{ww}
\end{bmatrix}\]
Note that while $\r{\Tilde{A}_{nf}} = \mathcal{O}(\epsilon)$, it is not the case with $\Tilde{S}_{nf}$. \[ \Tilde{S}_{nf} = S_{np}Q_{pf} = A_{pn}^TA_{pp}Q_{pf}+A_{nn}^TA_{np}Q_{pf} =A_{fn}^T \Tilde{A}_{ff} + A_{cn}^T\Tilde{A}_{cf}+A_{nn}^T \r{\Tilde{A}_{nf}}\]
Then, factorizing $f$ block through block Householder in $A$ gives us, 
\[H_f A Q_p = \begin{bmatrix}
R_{ff} & R_{fc} & {R_{fn}} & \\
& \hat{A}_{cc} & \hat{A}_{cn} & \\
& W_{cn}^{T} & A_{nn} & A_{nw} \\
& & A_{wn} & A_{ww}
\end{bmatrix} \quad (H_fAQ_p)^T(H_fAQ_p) = Q_p^TSQ_p \]
However, \Cref{lemma1} does not hold anymore, that is, $\|R_{fn}\| \ne \mathcal{O}(\epsilon)$. Since, $R_{fn}$ cannot be ignored, redefine $R_f$ as
\[R_f = \begin{bmatrix}
R_{ff} & R_{fc} & R_{fn} & \\
& I_c & & \\
& & I_n & \\
& & & I_w
\end{bmatrix}\]
\begin{align*}
    \hat{A} &=H_fAQ_pR_f^{-1} &\qquad \hat{S}&=R_f^{-T}Q_p^TSQ_pR_f \\
    &=\begin{bmatrix}
I_f &  &  & \\
& \hat{A}_{cc} & \hat{A}_{cn} & \\
& W_{cn}^{T} & A_{nn} & A_{nw} \\
& & A_{wn} & A_{ww} 
\end{bmatrix} &\qquad &=\begin{bmatrix}
I_f &  &   & \\
  & \hat{S}_{cc} & \hat{S}_{cn} &\\
  & \hat{S}_{nc} & \hat{S}_{nn} & S_{nw} \\ 
 & & S_{wn} & S_{ww}
\end{bmatrix}
\end{align*}
Also, $R_f^T = L_f + \mathcal{O(\epsilon)}$ where $L_f$ is the block Cholesky factor on elimination of $f$ block in $Q_p^TSQ_p$. Note that this modifies the $S_{nn}$ block which is not desired. For example, when an interface of a (parent) separator is sparsified, its children branches ($n = {n_1, n_2, \dots}$) can interact, breaking the ND ordering. While we have not affected the $A_{nn}$, we will notice the break in the ND ordering when we start factorizing the separators following the interface sparsification step. Although $\hat{A}_{lm}=\hat{A}_{ml} = 0$ for any two originally well-separated separators $l$, $m$ in the ND tree, Householder QR on separator $l$ will modify the columns of $m$ in $\hat{A}$ since $(\hat{A}^T\hat{A})_{lm}=\hat{S}_{lm} \ne \mathcal{O}(\epsilon)$ (see discussion in \Cref{sparseQR_S}). Hence this is not a good approach for sparsification of interfaces. 

\subsection*{Sparsification 2}

Instead consider a low rank approximation of $S_{np}^T$,
\[S_{np}^T = Q_{pp}W_{pn} = \begin{bmatrix}
Q_{pf} & Q_{pc}
\end{bmatrix} \begin{bmatrix}
W_{fn} \\
W_{cn}
\end{bmatrix} \text{ with } \|W_{fn}\|=\mathcal{O}(\epsilon)\] Defining $Q_p$ similarly, we have decoupled $f$ from $n$ in $G_{S}$ but not in $G_A$.
\[AQ_p = \begin{bmatrix}
\Tilde{A}_{ff} & \Tilde{A}_{fc} & A_{fn}  & \\
 \Tilde{A}_{cf} & \Tilde{A}_{cc} & A_{cn} &\\
 \Tilde{A}_{nf} & \Tilde{A}_{nc} & A_{nn} & A_{nw} \\ 
 & & A_{wn} & A_{ww}
\end{bmatrix}\quad Q_p^TSQ_p = \begin{bmatrix}
\Tilde{S}_{ff} & \Tilde{S}_{fc} & \mathcal{O}(\epsilon)  & \\
 \Tilde{S}_{cf} & \Tilde{S}_{cc} &  W_{cn} &\\
 \mathcal{O}(\epsilon) &  W_{cn}^{T} & S_{nn} & S_{nw} \\ 
 & & S_{wn} & S_{ww}
\end{bmatrix}\]
\[ \mathcal{O}(\epsilon) = \r{\Tilde{S}_{nf}} = S_{np}Q_{pf} = A_{pn}^TA_{pp}Q_{pf}+A_{nn}^TA_{np}Q_{pf} \]
$\r{\Tilde{S}_{nf}} = \mathcal{O}(\epsilon)$ does not necessarily imply that $A_{np}Q_{pf} = \Tilde{A}_{nf} = \mathcal{O}(\epsilon)$. Factorizing $f$ in $A$ through Householder QR, 
\[
H_f A Q_p = \begin{bmatrix}
R_{ff} & R_{fc} & \r{{R_{fn}}} & \\
& \hat{A}_{cc} & \hat{A}_{cn} & \\
& \hat{A}_{nc} & \hat{A}_{nn} & A_{nw} \\
& & A_{wn} & A_{ww}
\end{bmatrix} \qquad (H_fAQ_p)^T(H_fAQ_p) = Q_p^TSQ_p 
\]
We can show that $\|\r{R_{fn}}\| \leq \frac{\epsilon}{\sigma_{\text{min}}(A_p)}$, starting from the fact that $\r{\Tilde{S}_{nf}} = \mathcal{O}(\epsilon)$ and following the same procedure as \Cref{lemma2}. However $A_{nn}$ block is modified and can be dense. Hence, we will form connections between originally well-separated separators by \Cref{Coro1}. The elimination tree can become fully connected. However since $S_{nn}$ is not affected, the final sparsity pattern of $R$ is unchanged. Even so, this is still not a preferred method of sparsification. With this method the trailing matrix in $A$ becomes more dense after each step of interface sparsification, leading to a higher computational cost.

\subsection*{Sparsification 3}

We now describe spaQR. Keeping the drawbacks of the previous sparsification approaches in mind, we instead try to find an orthogonal transformation such that both $\hat{S}_{nf} \approx \mathcal{O}(\epsilon)$ and $\hat{A}_{nf} = \mathcal{O}(\epsilon)$. This is why the sparsification technique of the spaQR algorithm discussed in \Cref{spars_s} works. With a low rank appoximation of $\begin{bmatrix}
\hat{A}_{np}^{T} & \sigma\hat{A}_{pp}^{T}\hat{A}_{pn}
\end{bmatrix}$
\[ AQ_p = \begin{bmatrix}
\Tilde{A}_{ff} & \Tilde{A}_{fc} & A_{fn}  & \\
 \Tilde{A}_{cf} & \Tilde{A}_{cc} & A_{cn} &\\
 \mathcal{O}(\epsilon) & W_{cn}^{(1)T} & A_{nn} & A_{nw} \\ 
 & & A_{wn} & A_{ww}
\end{bmatrix} \qquad 
Q_p^TSQ_p = \begin{bmatrix}
\Tilde{S}_{ff} & \Tilde{S}_{fc} & \mathcal{O}(\epsilon)  & \\
 \Tilde{S}_{cf} & \Tilde{S}_{cc} &  W_{cn} &\\
 \mathcal{O}(\epsilon) &  W_{cn}^{T} & S_{nn} & S_{nw} \\ 
 & & S_{wn} & S_{ww}
\end{bmatrix}\]
\[\r{\Tilde{A}_{nf}} = A_{np}Q_{pf} = \r{W_{fn}^{(1)T}} = \mathcal{O}(\epsilon)\]
\[ \r{\Tilde{S}_{nf}} = S_{np}Q_{pf} = A_{pn}^TA_{pp}Q_{pf}+A_{nn}^TA_{np}Q_{pf} =  \frac{\r{W_{fn}^{(2)T}}}{\sigma}
+A_{nn}^T \r{W_{fn}^{(1)T}} \approx \mathcal{O}(\epsilon) \sigma_{\text{min}}(A_p)\]
On performing Householder on $f$ block, we get $\|\r{R_{fn}}\| \leq \epsilon $. However on performing Cholesky on $f$ block in $Q_p^TSQ_p$, we get,
\[\|\r{L_{nf}}\| =  \|\Tilde{S}_{nf}\Tilde{S}_{ff}^{-1}\| \leq \| \Tilde{S}_{nf}\| \|\Tilde{S}_{ff}^{-1}\| = \mathcal{O}(\epsilon) \sigma_{\text{min}}(A_p) \frac{1}{\lambda_{\text{min}}(S_{ff})} = \frac{\mathcal{O}(\epsilon)}{\sigma_{\text{min}}(A_p)} \]
since, $\lambda_{\text{min}}(S_{ff}) =\sigma_{\text{min}}(A_p)^2 $. The trailing matrices are, 
\begin{align*}
    \hat{A} &=H_fAQ_pR_f^{-1} &\qquad \hat{S}&=R_f^{-T}Q_p^TSQ_pR_f \\&=\begin{bmatrix}
I_f &  &  & \\
& \hat{A}_{cc} & \hat{A}_{cn} & \\
& W_{cn}^{(1)T} & A_{nn} & A_{nw} \\
& & A_{wn} & A_{ww}
\end{bmatrix} &\qquad  &= \begin{bmatrix}
I_f &  &   & \\
  & \hat{S}_{cc} & \hat{S}_{cn} &\\
  & {S}_{nc} & {S}_{nn} & S_{nw} \\ 
 & & S_{wn} & S_{ww}
\end{bmatrix}
\end{align*}

where the error in both $(\hat{A})_{nn}$ and $(\hat{S})_{nn}$ are $\mathcal{O}(\epsilon^2)$.

\section{Proof of \Cref{lemma1}}
\label{proof:lem1}

\sigchoice*
\begin{proof}
\begin{align*}
    \sigma Q_{pf}^T A_{pp}^TA_{pn} &= W_{fn}^{(2)} \\
    \sigma (A_{pp}Q_{pf})^T A_{pn} &= W_{fn}^{(2)} 
\end{align*}
\[A_{pp}Q_{pf} = \begin{bmatrix}
\Tilde{A}_{ff} \\
\Tilde{A}_{cf} \\
\mathcal{O}(\epsilon)\\
\\
\end{bmatrix}\]
And, \[H_{ff}^T A_{pp}Q_{pf} = R_{ff} + \mathcal{O}(\epsilon)\]
Then, 
\[ A_{pp}Q_{pf} = H_{ff}R_{ff} \]
This gives us, 
\begin{align*}
    W_{fn}^{(2)} &= \sigma (A_{pp}Q_{pf})^T A_{pn}
    = \sigma (H_{ff} R_{ff})^T A_{pn} \\
    &= \sigma R_{ff}^T (H_{ff}^T A_{pn})
    = R_{ff}^T R_{fn}\\
    R_{fn} &= \frac{1}{\sigma} R_{ff}^{-T}W_{fn}^{(2)}\\
    \|R_{fn}\|_{_2} &\leq \frac{1}{|\sigma|}\frac{1}{\sigma_{\text{min}}(R_{ff})} \epsilon
\end{align*}
Choosing, $\frac{1}{\sigma} = \sigma_{\text{min}}(A_p) \leq\sigma_{\text{min}}(R_{ff}) $, where $A_p = \begin{bmatrix}
A_{pp} \\
A_{np}
\end{bmatrix}$ proves the lemma.
\end{proof}

\section{Proof of \Cref{well_sep}}
\label{proof:wellsep}
\begin{lemma}
\label{lemma2}
Consider two separators $l$, $m$ such that $A_{lm}=A_{ml}= \mathcal{O}(\epsilon)$ and $(A^TA)_{lm}=(A^TA)_{ml}=\mathcal{O}(\epsilon)$. If $H_l$ is the block Householder transform on $A_{:l}$, then $H_lA_{:m} = A_{:m}+\mathcal{O(\epsilon)}$.
\end{lemma}
\begin{proof}
Consider the Householder transform on the columns of separator $l$. Let $x$ be the first column of $A_{:l}$. Then the Householder vector $v$ is defined as $v = x \pm \|x\|e_{l_1}$.  Since, 
\[ x^TA_{:m} = \mathcal{O}(\epsilon), e_{l_1}^TA_{:m} = \mathcal{O}(\epsilon) \]
we have, $v^TA_{:m} =\mathcal{O}(\epsilon)$. Then \[H_{l_1}A_{:m}=A_{:m}-\frac{2\mathcal{O}(\epsilon)}{v^Tv}v =A_{:m}+ \mathcal{O}(\epsilon)\mathbbm{1} \]
The lemma follows by doing an induction on the columns of $A_{:l}$
\end{proof}

\begin{corollary}
\label{Coro1}
For any two separators $l, m$ such that $(A^TA)_{lm}=(A^TA)_{ml} = \mathcal{O}(\epsilon)$ but $A_{lm}, A_{ml} \ne \mathcal{O}(\epsilon)$, then $\hat{A}_{:m} = H_lA_{:m} \ne A_{:m}+\mathcal{O}(\epsilon)$. However, $\hat{A}_{:m}^T \hat A_{:l} = \mathcal{O}(\epsilon)$, where $\hat{A}_{:l}= H_lA_{:l}$ 
\end{corollary}
\begin{proof}
This can be seen from the proof of \Cref{lemma2}; $c = e_{l1}^TA_{:m} \ne \mathcal{O}(\epsilon)$ and hence $v^TA_{:m} = \|x\|e_{l1}^TA_{:m} = c\|x\|$. Hence, $H_{l_1}A_{:m}=A_{:m}-\frac{2c\|x\|}{v^Tv}v$. 
\end{proof}

\wellsep*
\begin{proof}
Consider two separators $l$ and $m$ such that both are neighbors of an interface $p$ that is being sparsified. Let the separators $l$ and $m$ be such that $R_{lm}=0$ during direct QR on the matrix A. Trivially, this also implies that $A_{lm}=A_{ml}=0$, $(A^TA)_{lm}=(A^TA)_{ml}=0$. Then matrix block associated with the interface $p$ before sparsification has the following structure, 
\[A = \begin{bmatrix}
A_{pp} & A_{pl} & A_{pm} & \\
A_{lp} & A_{ll} & & A_{lw} \\
A_{mp} & & A_{mm} & A_{mw} \\
&A_{wl} & A_{wm} & A_{ww}
\end{bmatrix}\]
After sparsification and factorization of the `fine' nodes,
\[\hat{A} = H_fAQ_pR_f^{-1} = \begin{bmatrix}
I_f & & \r{R_{fl}} & \r{R_{fm}} & \\
& \hat{A}_{cc} & \hat{A}_{cl} & \hat{A}_{cm} & \\
& \hat{A}_{lc} & A_{ll}& & A_{lw} \\
& \hat{A}_{mc} & & A_{mm} & A_{mw} \\
& & A_{wl} &A_{wm} & A_{ww}
\end{bmatrix}\]
 The algorithm proceeds with matrix $\hat{A}$. While we still have, $\hat{A}_{lm}=\hat{A}_{ml}=\mathcal{O}(\epsilon^2)$, which is the error due to ignoring the term $\Tilde{A}_{nf} = W_{fn}^{(1)T} = \mathcal{O}(\epsilon)$, we need to show that $(\hat{A}^T \hat{A})_{lm}\approx 0$ 
\[ (\hat{A}^T\hat{A}) = (H_fAQ_pR_f^{-1})^T(H_fAQ_pR_f^{-1}) = (AQ_pR_f^{-1})^T(AQ_pR_f^{-1}) \]
Since, $R_f^{-1}$, $Q_p$ are operations that are only applied the columns of $p$, the columns of $l$ and $m$ are unaffected. This implies,
\[(\hat{A}^T\hat{A})_{lm} = (A^TA)_{lm} = 0 \]
\[ (\hat{A}^T\hat{A})_{lm} = \begin{bmatrix}
\r{R_{fl}} \\
\hat{A}_{cl}\\
A_{ll}\\
\\
A_{wl}
\end{bmatrix}^T \begin{bmatrix}
\r{R_{fm}} \\
\hat{A}_{cm}\\
\\
A_{mm}\\
A_{wm}
\end{bmatrix} = 0 \]
By \Cref{lemma1}, $\|R_{fl}\|_{_2}=\|R_{fm}\|_{_2}=\mathcal{O}(\epsilon)$ and are dropped. Then,
\[ (\hat{A}^T\hat{A})_{lm} = \begin{bmatrix}
\\
\hat{A}_{cl}\\
A_{ll}\\
\\
A_{wl}
\end{bmatrix}^T \begin{bmatrix}
 \\
\hat{A}_{cm}\\
\\
A_{mm}\\
A_{wm}
\end{bmatrix} = \mathcal{O}(\epsilon^2) \]
By \Cref{lemma2}, sparse QR on separator $l$ (or $m$) will not affect separator $m$ up to a tolerance of $\mathcal{O}(\epsilon^2)$. And since interface sparsification does not affect the non-neighbor blocks, sparsification of $l$ or $m$ will not affect the other (up to the same tolerance of $\mathcal{O}(\epsilon)$). Hence the algorithm can proceed without affecting the elimination tree. 
\end{proof}


\section*{Acknowledgements}
The computing for this project was performed on the Sherlock research cluster, hosted at Stanford University. We thank Stanford University and the Stanford Research Computing Center for providing the computational resources and support that contributed to this research. This work was partly funded by a grant from Sandia National Laboratories (Laboratory Directed Research and Development [LDRD]) entitled ``Hierarchical Low-rank Matrix Factorizations,'' and a grant from the National Aeronautics and Space Administration (NASA, agreement \#80NSSC18M0152). We thank L\'eopold Cambier, Erik G.\ Boman and Juan Alonso for the numerous discussions. We also thank Jordi Feliu-F\'aba and Steven Brill from Stanford ICME for valuable discussions.

\bibliographystyle{siamplain}
\bibliography{main}

\begin{thebibliography}{10}

\bibitem{mumps}
{\sc P.~R. Amestoy, I.~S. Duff, J.-Y. L'Excellent, and J.~Koster}, {\em Mumps:
  A general purpose distributed memory sparse solver}, in Applied Parallel
  Computing. New Paradigms for HPC in Industry and Academia, T.~S{\o}revik,
  F.~Manne, A.~H. Gebremedhin, and R.~Moe, eds., Berlin, Heidelberg, 2001,
  Springer Berlin Heidelberg, pp.~121--130.

\bibitem{H_QR}
{\sc P.~Benner and T.~Mach}, {\em On the qr decomposition of h-matrices},
  Computing, 88 (2010), \url{https://doi.org/10.1007/s00607-010-0087-y}.

\bibitem{Boisvert1997}
{\sc R.~F. Boisvert, R.~Pozo, K.~Remington, R.~F. Barrett, and J.~J. Dongarra},
  {\em Matrix Market: a web resource for test matrix collections}, Springer US,
  Boston, MA, 1997, pp.~125--137,
  \url{https://doi.org/10.1007/978-1-5041-2940-4_9},
  \url{https://doi.org/10.1007/978-1-5041-2940-4_9}.

\bibitem{leopold_matrixgen}
{\sc L.~Cambier}, {\em Matrix gen}.
\newblock \url{https://github.com/leopoldcambier/MatrixGen}.

\bibitem{2019arXiv190102971C}
{\sc L.~Cambier, C.~Chen, E.~Boman, S.~Rajamanickam, R.~Tuminaro, and
  E.~Darve}, {\em An algebraic sparsified nested dissection algorithm using
  low-rank approximations}, SIAM Journal on Matrix Analysis and Applications,
  41 (2020), pp.~715--746, \url{https://doi.org/10.1137/19M123806X}.

\bibitem{Catalyurek_hypergraph-partitioningbased}
{\sc U.~V. Catalyurek and C.~Aykanat}, {\em Hypergraph-partitioning based
  decomposition for parallel sparse-matrix vector multiplication}, IEEE Trans.
  on Parallel and Distributed Computing, 10, pp.~673--693.

\bibitem{Chow1997ExperimentalSO}
{\sc E.~Chow and Y.~Saad}, {\em Experimental study of ilu preconditioners for
  indefinite matrices}, Journal of Computational and Applied Mathematics, 86
  (1997), pp.~387--414.

\bibitem{suitesparse}
{\sc T.~A. Davis and Y.~Hu}, {\em The university of florida sparse matrix
  collection}, ACM Trans. Math. Softw., 38 (2011),
  \url{https://doi.org/10.1145/2049662.2049663},
  \url{https://doi.org/10.1145/2049662.2049663}.

\bibitem{ZoltanHypergraphIPDPS06}
{\sc K.~D. Devine, E.~G. Boman, R.~T. Heaphy, R.~H. Bisseling, and U.~V.
  Catalyurek}, {\em Parallel hypergraph partitioning for scientific computing},
  IEEE, 2006.

\bibitem{C:LaBRI::CIMI15}
{\sc M.~Faverge, G.~Pichon, P.~Ramet, and J.~Roman}, {\em On the use of
  h-matrix arithmetic in pastix: a preliminary study}, in Workshop on Fast
  Solvers, Toulouse, France, June 2015,
  \url{http://www.labri.fr/~ramet/restricted/cimi15.pdf}.

\bibitem{FeliuFab2018RecursivelyPH}
{\sc J.~Feliu-Fabà, K.~Ho, and L.~Ying}, {\em Recursively preconditioned
  hierarchical interpolative factorization for elliptic partial differential
  equations}, Communications in Mathematical Sciences, 18 (2020), pp.~91--108,
  \url{https://doi.org/10.4310/CMS.2020.v18.n1.a4}.

\bibitem{feliufaba2020hierarchical}
{\sc J.~Feliu-Fabà and L.~Ying}, {\em Hierarchical interpolative factorization
  preconditioner for parabolic equations}, 2020,
  \url{https://arxiv.org/abs/2004.05566}.

\bibitem{Ghysels2016AnEM}
{\sc P.~Ghysels, X.~S. Li, F.-H. Rouet, S.~Williams, and A.~Napov}, {\em An
  efficient multicore implementation of a novel hss-structured multifrontal
  solver using randomized sampling}, SIAM J. Scientific Computing, 38 (2016).

\bibitem{10.5555/248979}
{\sc G.~H. Golub and C.~F. Van~Loan}, {\em Matrix Computations (3rd Ed.)},
  Johns Hopkins University Press, USA, 1996.

\bibitem{hmatrix_2}
{\sc L.~Grasedyck and W.~Hackbusch}, {\em Construction and arithmetics of
  h-matrices}, Computing, 70 (2003), p.~295–334,
  \url{https://doi.org/10.1007/s00607-003-0019-1},
  \url{https://doi.org/10.1007/s00607-003-0019-1}.

\bibitem{FMM_1}
{\sc L.~Greengard and V.~Rokhlin}, {\em A fast algorithm for particle
  simulations}, J. Comput. Phys., 135 (1997), p.~280–292,
  \url{https://doi.org/10.1006/jcph.1997.5706},
  \url{https://doi.org/10.1006/jcph.1997.5706}.

\bibitem{greengard_rokhlin_1997}
{\sc L.~Greengard and V.~Rokhlin}, {\em A new version of the fast multipole
  method for the laplace equation in three dimensions}, Acta Numerica, 6
  (1997), p.~229–269, \url{https://doi.org/10.1017/S0962492900002725}.

\bibitem{Grigori_hypergraph-basedunsymmetric}
{\sc L.~Grigori, E.~G. Boman, S.~Donfack, and T.~A. Davis}, {\em
  Hypergraph-based unsymmetric nested dissection ordering for sparse lu
  factorization}.

\bibitem{hmatrix_1}
{\sc W.~Hackbusch}, {\em A sparse matrix arithmetic based on h-matrices. part
  i: Introduction to h-matrices}, Computing, 62 (1999), p.~89–108,
  \url{https://doi.org/10.1007/s006070050015},
  \url{https://doi.org/10.1007/s006070050015}.

\bibitem{hmatrix_3}
{\sc W.~Hackbusch and B.~Khoromskij}, {\em A sparse h-matrix arithmetic:
  general complexity estimates}, Journal of Computational and Applied
  Mathematics, 125 (2000), pp.~479 -- 501,
  \url{https://doi.org/https://doi.org/10.1016/S0377-0427(00)00486-6},
  \url{http://www.sciencedirect.com/science/article/pii/S0377042700004866}.
\newblock Numerical Analysis 2000. Vol. VI: Ordinary Differential Equations and
  Integral Equations.

\bibitem{Hackbusch2000ASH}
{\sc W.~Hackbusch and B.~N. Khoromskij}, {\em A sparse h -matrix arithmetic:
  general complexity estimates}, 2000.

\bibitem{Hestenes&Stiefel:1952}
{\sc M.~R. Hestenes and E.~Stiefel}, {\em Methods of conjugate gradients for
  solving linear systems}, Journal of research of the National Bureau of
  Standards, 49 (1952), pp.~409--436.

\bibitem{Ho2016HierarchicalIF}
{\sc K.~L. Ho and L.~Ying}, {\em Hierarchical interpolative factorization for
  elliptic operators: Integral equations}, 2016.

\bibitem{hsl_mc64}
{\sc {HSL(2013)}}, {\em A collection of fortran codes for large scale
  scientific computation}, \url{http://www.hsl.rl.ac.uk}.

\bibitem{jennings}
{\sc A.~Jennings and M.~A. Ajiz}, {\em Incomplete methods for solving $a^t ax =
  b$}, SIAM J. Sci. Stat. Comput., 5 (1984), p.~978–987,
  \url{https://doi.org/10.1137/0905067}, \url{https://doi.org/10.1137/0905067}.

\bibitem{Karypis1998HmetisAH}
{\sc G.~Karypis and V.~Kumar}, {\em Hmetis: a hypergraph partitioning package},
  1998.

\bibitem{klockiewicz2020second}
{\sc B.~Klockiewicz, L.~Cambier, R.~Humble, H.~Tchelepi, and E.~Darve}, {\em
  Second order accurate hierarchical approximate factorization of sparse spd
  matrices}, 2020, \url{https://arxiv.org/abs/2007.00789}.

\bibitem{elimination_tree}
{\sc J.~W.~H. Liu}, {\em The role of elimination trees in sparse
  factorization}, SIAM J. Matrix Anal. Appl., 11 (1990), p.~134–172,
  \url{https://doi.org/10.1137/0611010}, \url{https://doi.org/10.1137/0611010}.

\bibitem{Manteuffel1980AnIF}
{\sc T.~Manteuffel}, {\em An incomplete factorization technique for positive
  definite linear systems}, Mathematics of Computation, 34 (1980),
  pp.~473--497.

\bibitem{citeulike:10745617}
{\sc C.~C. Paige and M.~A. Saunders}, {\em {Solution of Sparse Indefinite
  Systems of Linear Equations}}, SIAM Journal on Numerical Analysis, 12 (1975),
  pp.~617--629, \url{https://doi.org/10.1137/0712047},
  \url{http://dx.doi.org/10.1137/0712047}.

\bibitem{blr_pastix}
{\sc G.~{Pichon}, E.~{Darve}, M.~{Faverge}, P.~{Ramet}, and J.~{Roman}}, {\em
  Sparse supernodal solver using block low-rank compression}, in 2017 IEEE
  International Parallel and Distributed Processing Symposium Workshops
  (IPDPSW), 2017, pp.~1138--1147.

\bibitem{lorasp1}
{\sc H.~Pouransari, P.~Coulier, and E.~Darve}, {\em Fast hierarchical solvers
  for sparse matrices using extended sparsification and low-rank
  approximation}, SIAM Journal on Scientific Computing, 39 (2017),
  pp.~A797--A830, \url{https://doi.org/10.1137/15M1046939}.

\bibitem{Saad1988PreconditioningTF}
{\sc Y.~Saad}, {\em Preconditioning techniques for nonsymmetric and indefinite
  linear systems}, 1988.

\bibitem{Saad1994ILUTAD}
{\sc Y.~Saad}, {\em Ilut: A dual threshold incomplete lu factorization},
  Numerical Lin. Alg. with Applic., 1 (1994), pp.~387--402.

\bibitem{Saad1986GMRESAG}
{\sc Y.~Saad and M.~H. Schultz}, {\em Gmres: a generalized minimal residual
  algorithm for solving nonsymmetric linear systems}, 1986.

\bibitem{Schmitz2012AFD}
{\sc P.~G. Schmitz and L.~Ying}, {\em A fast direct solver for elliptic
  problems on general meshes in 2d}, J. Comput. Phys., 231 (2012),
  pp.~1314--1338.

\bibitem{Xia2013EfficientSM}
{\sc J.~Xia}, {\em Efficient structured multifrontal factorization for general
  large sparse matrices}, SIAM J. Scientific Computing, 35 (2013).

\bibitem{Xia2009SuperfastMM}
{\sc J.~Xia, S.~Chandrasekaran, M.~Gu, and X.~S. Li}, {\em Superfast
  multifrontal method for large structured linear systems of equations}, SIAM
  J. Matrix Analysis Applications, 31 (2009), pp.~1382--1411.

\bibitem{Xia2017EffectiveAR}
{\sc J.~Xia and Z.-X. Xing}, {\em Effective and robust preconditioning of
  general spd matrices via structured incomplete factorization}, SIAM J. Matrix
  Analysis Applications, 38 (2017), pp.~1298--1322.

\bibitem{lorasp2}
{\sc K.~Yang, H.~Pouransari, and E.~Darve}, {\em Sparse hierarchical solvers
  with guaranteed convergence}, International Journal for Numerical Methods in
  Engineering,  (2016), \url{https://doi.org/10.1002/nme.6166}.

\bibitem{atalyrek1999HypergraphMF}
{\sc {\"U}.~V. Çataly{\"u}rek}, {\em Hypergraph models for sparse matrix
  partitioning and reordering}, 1999.

\bibitem{atalyrek2011PaToHT}
{\sc {\"U}.~V. Çataly{\"u}rek and C.~Aykanat}, {\em Patoh (partitioning tool
  for hypergraphs)}, in Encyclopedia of Parallel Computing, 2011.

\end{thebibliography}

\end{document}